\documentclass{amsart}
\usepackage{amsmath,amssymb,amsthm,mathrsfs,amscd}
\usepackage{indentfirst}
\usepackage{amsrefs} 
\usepackage{enumerate}
\usepackage{fullpage}
\usepackage{mathptmx}
\usepackage{cleveref}
\usepackage{arydshln}  
\usepackage{xcolor}

\theoremstyle{plain}
\newtheorem{thm}{Theorem}[section]
\newtheorem{lem}[thm]{Lemma}
\newtheorem{prop}[thm]{Proposition}
\newtheorem{cor}[thm]{Corollary}
\newtheorem{ques}[thm]{Question}

\newtheorem*{them*}{Theorem}
\newtheorem*{prop*}{Proposition}

\theoremstyle{definition}
\newtheorem{dfn}[thm]{Definition}
\newtheorem{rmk}[thm]{Remark}

\newtheorem{set}[thm]{Setting}
\newtheorem{exam}[thm]{Example}

\Crefname{thm}{Theorem}{Theorems}
\Crefname{lem}{Lemma}{Lemmas}

\DeclareMathOperator{\HH}{H}
\DeclareMathOperator{\gr}{{gr}}

\newcommand{\kk}{{\mathbf k}}
\newcommand{\ee}{{\mathbf e}}

\newcommand{\ZZ}{{\mathbb {Z}}}

\newcommand{\RR}{{\mathbb {R}}}

\newcommand{\MM}{{\mathfrak {M}}}
\newcommand{\MN}{{\mathfrak {N}}}
\newcommand{\p}{\mathfrak p }
\newcommand{\n}{\mathfrak n}
\newcommand{\m}{\mathfrak m}

\DeclareMathOperator{\he}{ht}

\DeclareMathOperator{\core}{core}
\DeclareMathOperator{\Hom}{Hom}
\DeclareMathOperator{\Tor}{Tor}
\DeclareMathOperator{\Ext}{Ext}
\DeclareMathOperator{\sExt}{^*Ext}

\DeclareMathOperator{\Soc}{Soc}
\DeclareMathOperator{\indeg}{indeg}

\DeclareMathOperator{\Min}{Min}
\DeclareMathOperator{\Ass}{Ass}
\DeclareMathOperator{\depth}{depth}

\DeclareMathOperator{\Spec}{Spec}

\DeclareMathOperator{\sHom}{^*Hom}

\DeclareMathOperator{\CC}{\mathcal C}

\DeclareMathOperator{\lcm}{lcm}
\DeclareMathOperator{\ord}{ord}

\DeclareMathOperator{\NP}{{NP}}
\DeclareMathOperator{\relint}{relint}

\newcommand{\ti}[1]{{\textit{#1}}}

\begin{document}  
\title{Quasi-Gorensteinness of extended Rees algebras}
\author{Youngsu Kim}
\address{Department of Mathematics, University of California, Riverside,
California, 92521 U.S.A}
\email{youngsu.kim@ucr.edu}

\date{\today} 

\begin{abstract}
Let $R$ be a Noetherian local ring and $I$ an $R$-ideal. It is well-known that if the associated graded ring $\gr_I(R)$ is Cohen-Macaulay (Gorenstein), then so is $R$, but the converse is not true in general. In this paper we investigate the Cohen-Macaulayness and Gorensteinness of the associated graded ring $\gr_I(R)$ under the hypothesis of the extended Rees algebra $R[It,t^{-1}]$ is quasi-Gorenstein or the associated graded ring $\gr_I(R)$ is a domain. 
\end{abstract}

\maketitle

\section{Introduction}
A Noetherian ring having a canonical module is called \textit{quasi-Gorenstein} if it is  locally isomorphic to the canonical module. Clearly, a ring is Gorenstein if and only if it is quasi-Gorenstein and  Cohen-Macaulay. Murthy \cite{Mu64} showed that a Cohen-Macaulay UFD having a canonical module is Gorenstein. In general, the UFD property implies quasi-Gorensteinness if the ring has a canonical module. There exists a complete UFD having a canonical module which is not Cohen-Macaulay, see \cite[Theorem 5.8]{FK74}. This shows that a quasi-Gorenstein ring needs not to be  Gorenstein in general. Surprisingly, the quasi-Gorenstein property implies the Gorensteinness for some classes of extended Rees algebras. In this regard, Heinzer, M.-K. Kim, and Ulrich posed the following question.

\begin{ques}[{\cite[Question 4.11]{HKU05}}]\label{Question1}
Let $(R,\m)$ be a local Gorenstein ring and let $I$ be an $\m$-primary ideal. Is the extended Rees algebra $R[It,t^{-1}]$ Gorenstein $($equivalently Cohen-Macaulay$)$ if it is quasi-Gorenstein?
\end{ques}

When the dimension of the ring $R$ is $1$, \Cref{Question1} has an affirmative answer because $R[It,t^{-1}]$ has dimension $2$ in this case and quasi-Gorenstein rings satisfy Serre's condition $(S_2)$. 
The authors showed that \Cref{Question1} has an affirmative answer when $R$ is a $2$-dimensional pseudo-rational ring \cite[Cor. 4.12]{HKU05}. The general case still remains open. However, if one removes the condition of $I$ being $\m$-primary, then there exists an extended Rees algebra which is a UFD (hence quasi-Gorenstein), but not Gorenstein  \cite[Example 4.7]{HH92}. 
In \Cref{sec:q-gor} we provide an affirmative answer to \Cref{Question1}  if $I$ is an almost complete intersection under the additional assumption that the index of nilpotency and the reduction number of $I$ coincide (which is a necessary condition for $R[It,t^{-1}]$ to be Cohen-Macaulay), see \Cref{thm:qGor2}. We are also able to treat the case when $I$ is a monomial ideal in a polynomial ring in $d$-variables and $I$ has a $d$-generated monomial reduction, see \Cref{thm:qGor1}. The latter condition, $I$ having such a reduction, is equivalent to the condition that $I$ has only one Rees valuation.\\

For monomial ideals in a polynomial ring, the normalization of the extended Rees algebra is Cohen-Macaulay \cite{Ho73}. Therefore, these normalizations are quasi-Gorenstein if and only if they are Gorenstein. In \cite{HKU11} the authors characterized the Gorenstein property of normalized extend Rees algebras of a monomial ideal when the ideal of finite colength has only one Rees valuation. Recall that the Rees valuations of such monomial ideals correspond to the bounded half spaces defining the Newton Polyhedron of $I$. 
We are able to remove the condition of having one Rees valuation, and provide a characterization of the Gorenstein property in terms of the half spaces defining the Newton polyhedron of the given ideal.\\

Since $\gr_I(R) \cong R[It,t^{-1}]/(t^{-1})$ and $t^{-1}$ is a homogeneous non zero-divisor, the associated graded ring $\gr_I(R)$ is Cohen-Macaulay (or Gorenstein) if and only if the extended Rees algebra $R[It,t^{-1}]$ is. Hochster \cite[p.\ 55, Proposition]{Ho73} and Herzog, Simis, Vasconcelos \cite[Proposition 1.1]{HSV87} showed that if $R$ is a local Gorenstein ring and the associated graded ring $\gr_I(R)$ is a domain, then $R[It,t^{-1}]$ is quasi-Gorenstein. Hence an affirmative answer to \Cref{Question1} would imply that $\gr_I (R)$ is Cohen-Macaulay if it is a domain. This version of the question is meaningful even when the ambient ring $R$ is not Gorenstein. 

\begin{ques}\label{Question2}
Let $(R, \m)$ be a local Cohen-Macaulay ring. Is $\gr_{\m}(R)$ Cohen-Macaulay if it is a domain?
\end{ques}
The question has an affirmative answer when $R$ is a complete intersection ring of embedding codimension at most $2$, i.e., $\widehat{R} \cong S/I$ where $(S, \n)$ is a regular local ring and $I$ is a complete intersection ideal of height at most $2$. It is natural to ask if the question has an affirmative answer when the ideal $I$ is generated by $3$ elements. In \Cref{sec:grDomain} we prove that this is indeed the case if in addition $I \nsubseteq \n^5$ (\Cref{thm:integralGr}). \\

We already mentioned that if $\gr_I (R)$ is Cohen-Macaulay, then $R$ is Cohen-Macaulay. Recall that the Cohen-Macaulyness of a ring can be characterized by Serre's condition $(S_i)$ for every $i$.
 
\begin{ques}\label{Question3}
Let $R$ be a Noetherian ring and $I$ be an $R$-ideal. If $\gr_I(R)$ satisfies Serre's condition $(S_i)$ $($or $(R_i))$, then does $R[It,t^{-1}]$ satisfy the same condition?
\end{ques}

We give a positive answer to \Cref{Question3} for proper ideals in a universally catenary equidimensional local ring (\Cref{thm:Serre}).\\

The outline of the paper is as follows: In \Cref{sec:prel} we set up the notation. In Section 3 we start by introducing the graded canonical module of $\mathbb{Z}$-graded rings and study basic properties related to the quasi-Gorensteinness of extended Rees algebras. The two main theorems, \Cref{thm:qGor1,thm:qGor2}, are proved in this section. 
In addition, a result on the $\textbf{a}$-invariant of the extended Rees algebra (\Cref{thm:a-inv}) and the core of powers of the ideal (\Cref{thm:core}) are presented. 
In Section 4 we provide a characterization of the Gorensteinness of normalized extended Rees algebras of finite colength monomial ideals in a polynomial ring. This is a generalization of \cite[Theorem 5.6]{HKU11}.  
In Section 5 we discuss \Cref{Question2} in detail, and provide a positive answer in the case of almost complete intersections of codimension $2$ (\Cref{thm:integralGr}). 
In Section 6 we give a positive answer to \Cref{Question3} when the ring in question is a universally catenary equidimensional local ring and the ideal is not a unit ideal (\Cref{thm:Serre}).

\section*{Acknowledgements}
This note is part of the author's Ph.\ D.\ thesis at Purdue University under the direction of Bernd Ulrich. 
The author would like give special thanks to his academic advisor Bernd Ulrich for his tremendous support. 
The author is also grateful to Bill Heinzer for stimulating conversations and suggestions, and the author would like to thank Sam Huckaba for the discussion of an example in his paper. Lastly, the author thanks for the anonymous referee for his/her helpful comments.

\section{Preliminaries}\label{sec:prel} 
All rings are commutative Noetherian with unity. Let $(R,\m)$ be a local ring and $M$ a finitely generated $R$-module. Let $\mu(M)$ denote the minimal number of generators of $M$. By $\Min(M)$ and $\Ass(M)$ we denote the set of minimal primes and associated primes of $M$, respectively. 
For $R$-ideals $J \subseteq I$, the ideal $J$ is called a \textit{reduction} of $I$ if there exists a non-negative integer $n$ such that $J I^n = I^{n+1}$, the smallest such integer is called the \textit{reduction number} of $I$ with respect to $J$, denoted by $r_J(I)$, and let $r(I) := \min \{ r_J(I) \mid J \hbox{ a reduction of } I \}$. 
For an $r \times s$ matrix $\Phi$ with entries in $R$, we write $I_n (\Phi)$ for the ideal generated by the $n \times n$ minors of $\Phi$. By convention, we set $I_n(\Phi) = R$ for $n \le 0$ and we set $I_n (\Phi) = 0$ for $n > \min \{ r,s \}$ .\\

For an ideal $I \subset R$, we write
\[
R[It] = \oplus_{i \ge 0} I^i t^i, R[It,t^{-1}] = \oplus_{i \in \mathbb{Z} } I^i t^i, \hbox{ and }\gr_I(R ) = \oplus_{i \ge 0} I^i/I^{i+1},
\]
for the \textit{Rees algebra}, the \textit{extended Rees algebra}, and the  \textit{associated graded ring} of $R$ with respect to the ideal $I$, respectively. Sometimes, they are also called blowup algebras. These are the rings which appear in the construction of blowing up an affine variety along a closed subvariety in algebraic geometry. \\

Let $\omega_R$ denote a \textit{canonical module} of $R$ if it exists. Here the \textit{canonical module} is the dualizing module in the sense of Grothendieck's local duality theorem. For instance, when $(R,\m)$ is a complete local ring of dimension $d$, then a canonical module $\omega_R$ of $R$ is $\Hom_R( \HH_{\m}^d ( R ) , E_R(R/ \m))$, where $\HH_{\m}^d (-)$  and $E_R(-)$ denote the $d$th local cohomology with respect to $\m$ and the injective envelope, respectively. 
A local ring $R$ having a canonical module $\omega_R$ is called \textit{Gorenstein} if $R$ is Cohen-Macaulay and $R$ is the canonical module $\omega_R$. The second property $R \cong \omega_R$ can be isolated, and it is called the quasi-Gorenstein property. That is, a local ring is \textit{quasi-Gorenstein} if it is isomorphic to the canonical module $\omega_R$. \\

Let $S$ be a Noetherian $\mathbb{Z}$-graded ring with unique homogeneous maximal ideal $\MM$ and assume that $\MM$ is maximal.  For a graded module $M$, let $[M]_i$ denote the $i$th graded piece of $M$. The \textbf{a}-\textit{invariant} of $S$, denoted by $a(S)$, is $\max \{ i \in \mathbb{Z} \mid [^*\Soc (\HH_{\MM}^{d}(S))]_i \neq 0 \}$. Here $^*  \! \Soc(M) = 0 :_M \MM$ for any graded $S$-module $M$. If $S$ has a graded canonical module $\omega_S$, then by graded local duality we have $a(S) = - \min \{ i \in \mathbb{Z} \mid [\omega_S/ \MM\omega_S]_i \neq 0 \}$.  If $S$ is positively graded, then this number is $\max \{ i \in \mathbb{Z} \mid [\HH_{\MM}^{d}(S)]_i \neq 0 \}$. \\

In the sequel we use the book by Bruns and Herzog \cite{BH} as a reference for basic definitions and terminologies.

\section{Quasi-Gorensteinness of extended Rees algebras}\label{sec:q-gor}
\section*{Graded canonical modules}

Let $R$ be a Noetherian $\ZZ$-graded ring with unique maximal homogeneous ideal $\MM$. Then the subring $R_0$ is local. We write $\m$ for the maximal ideal of $R_0$. Let $E_{R_0} (R_0 / \m)$ be the injective envelope of the residue field of $R_0$. 
For a homogeneous prime ideal $\p$ of $R$, we write $R_{(\p)} := S^{-1} R$ where $S$ is the set of the homogeneous elements of $R$ which are not in $\p$. Observe that $R_{(\p)}$ is $\ZZ$-graded and has unique maximal homogeneous ideal $\p R_{(\p)}$. 
For $\ZZ$-graded $R$-modules $M$ and $N$, let $\sHom_R (M, N)$ denote 
the $R$-submodule of $\Hom_R (M, N)$ generated by the homogeneous $R$-linear maps of arbitrary degree from $M$ to $N$. For a finitely generated $\ZZ$-graded $R$-module $M$, let $\widehat{M}$ denote the tensor product $M \otimes_{R_0} \widehat{R_0}^{\m}$. 
We say $M$ is \ti{$^*$-complete} if $M \cong \widehat{M}$ by the natural isomorphism. 
In particular, $\widehat{R}$ is $^*$-complete, and $R$ is $^*$-complete if and only if $R_0$ is complete.

\begin{dfn} Let $d = \dim R_{\MM}$. 
A finitely generated graded $R$-module $\omega_R$ is called a \ti{graded canonical module} of $R$ if 
\[
\widehat{\omega_R}  \cong \sHom_{R_0} ( {\HH}^{d}_{\MM}(R), E_{R_0} (R_0 / \m)) 
\]
as graded $R$-modules.
\end{dfn}

\begin{rmk}\label{sCompletion}
Let $R$ be a Noetherian $\ZZ$-graded ring with a unique maximal homogeneous ideal. 
\begin{enumerate}[$(a)$]
\item If $R$ is $^*$-complete, then $R$ has a graded canonical module. 
\item {{\cite[Exercise 7.5]{Ei95}}} For finitely generated $\ZZ$-graded $R$-modules $M$ and $N$, if $\widehat{M} \cong \widehat{N}$, then $M \cong N$. 
\end{enumerate}
\end{rmk}

\begin{lem}[cf. {\cite[Corollary 3.6.14]{BH}}]\label{regHomSeq} 
Let $R$ be a $\ZZ$-graded Cohen-Macaulay ring with a unique maximal homogeneous ideal. Assume that $R$ has a graded canonical module $\omega_R$. If $\textbf{\underline{x}} = x_1,\dots, x_n$ form a homogeneous regular sequence on $R$, then $\textbf{\underline{x}}$ form a regular sequence on $\omega_R$ and we have
\[
\omega_{R/(\textbf{\underline{x}} )} \cong (\omega_{R} / \textbf{\underline{x}} \, \omega_R ) (\sum_{i=1}^n \deg (x_i)).
\]
\end{lem}
\begin{proof}
It suffices to show the statement when $n = 1$.
Let $\MM$ be the unique maximal homogeneous ideal, $d = \dim R_{\MM}$, and $\m$  the maximal ideal of $R_0$. 
By \Cref{sCompletion}(b), we may assume that $R$ is $*$-complete (hence so is $R/xR$).
Since $R$ and $R/(x)$ are Cohen-Macaulay, the exact sequence
\[
0 \to R ( - \deg (x) ) \stackrel{\cdot x}{\to} R \to R/ (x) \to 0
\]
 induces the exact sequence 
\[
0 \to {\HH}^{d-1}_{\MM}(R/(x)) \to {\HH}^{d}_{\MM}(R) ( - \deg (x)) \stackrel{\cdot x}{\to} {\HH}^{d}_{\MM}(R)  \to 0.
\]
Taking $\sHom_{R_0} ( -, E_{R_0} (R_0 / \m))$ we obtain the exact sequence
\begin{equation}\label{eq:regHomSeq}
0 \to \omega_R \stackrel{\cdot x}{\to} \omega_R ( \deg(x)) \to \sHom_{R_0} ({\HH}^{d-1}_{\MM}(R/(x)), E_{R_0} (R_0 / \m))  \to 0.
\end{equation}
This shows that $x$ is a non zerodivisor of $\omega_R$. Hence we are done once we have shown the isomorphism  $\omega_{R/(x)} \cong \sHom_{R_0}({\HH}^{d-1}_{\MM}(R/(x)), E_{R_0} (R_0 / \m))$. From the ring homomorphism $R \to R/(x)$, we obtain a surjective ring homomorphism $R_0 \to [R/(x)]_0$ of local rings. We write $(\overline{R_0}, \overline{\m})$ for the local ring $[R/(x)]_0$. 
By \cite[Exercise 13]{HuLLC} we have $E_{\overline{R_0}} (\overline{R_0} / \overline{ \m})) \cong \Hom_{R_0} (\overline{R_0}, E_{R_0} (R_0 / \m))$. 
Therefore, by the hom-tensor adjointness we obtain
\begin{align*}
\sHom_{\overline{R_0}}(- ,  E_{\overline{R_0}} (\overline{R_0} / \overline{ \m}))  
	&\cong \sHom_{\overline{R_0}}(- , \sHom_{R_0} (\overline{R_0}, E_{R_0} (R_0 / \m)) ) \\
	&\cong \sHom_{R_0}(- \otimes_{\overline{R_0}} \overline{R_0},  E_{R_0} (R_0 / \m)) \\
	&\cong \sHom_{R_0}(- ,  E_{R_0} (R_0 / \m))
\end{align*}
for any $R/ (x)$-module in the first variable.
Since ${\HH}^{d-1}_{\MM}(R/(x)) \cong {\HH}^{d-1}_{\MM/(x)}(R/(x))$, this shows the statement. 
\end{proof}

The following lemma is a partial converse of \Cref{regHomSeq}.
\begin{lem}\label{ConvRegHomSeq} 
Let $R$ be a $\ZZ$-graded Cohen-Macaulay ring with a unique maximal homogeneous ideal. 
Assume that $\textbf{\underline{x}} = x_1,\dots, x_n$ form a homogeneous regular sequence, and $R/ (x_1,\dots, x_n)$ has a graded canonical module $\omega_{R/ (x_1,\dots, x_n)}$. Write $\rho = \sum_{i = 1}^n \deg(x_i)$.
If $\omega_{R/ (x_1,\dots, x_n)} \cong (R/ (x_1,\dots, x_n))(a)$ for some $a \in \ZZ$, then $\omega_R \cong R(a - \rho)$. 
\end{lem}
\begin{proof}
It suffices to show the statement when $n=1$. By \Cref{sCompletion}(a) $\omega_{\widehat{R}}$ exists, and by \Cref{sCompletion}(b) it suffices to show that $\widehat{R(a - \rho)} \cong \omega_{\widehat{R}}$. Hence we may assume that $R$ is $^*$-complete. By \Cref{regHomSeq} we have $ \omega_R / x \omega_R ( \deg(x)) \cong  \omega_{R/(x)} \cong R/(x) (a)$. By Nakayama's lemma we see that $\omega_R$ is a cyclic $R$-module. Consider the exact sequence
\begin{equation}\label{eq:faithful}
0 \to K \to R(a - \deg(x)) \to \omega_R \to 0.
\end{equation}
We tensor \Cref{eq:faithful} with $R/(x)$. By \Cref{regHomSeq} $x$ is a non zerodivisor on $\omega_R$. This implies that $\Tor_1^R ( \omega_R, R/(x)) = 0$, i.e., \Cref{eq:faithful} remains exact after applying $- \otimes_R R/(x)$. Since $R/(x) (a - \deg (x) ) \cong \omega_R / x \omega_R$, one has $K / xK = 0$. By Nakayama's lemma we obtain $K = 0$. Indeed this implies $R(a- \deg(x)) \cong \omega_R$. 
\end{proof}

\begin{thm}[cf. {\cite[Theorem 5.12]{HK71}}]\label{GddHK5.12} Let $R$ and $S$ be Noetherian $\ZZ$-graded rings with unique maximal homogeneous ideals $\MM$ and $\MN$, respectively. Let $\phi: S \to R$ be a graded ring homomorphism. 
Assume that $R$ is a finitely generated $S$-module, $\phi(S_0) = R_0$, and $S$ is Cohen-Macaulay. 
Write $\dim S_{\MN} = n$ and $\dim R_{\MM} = d$. If $S$ has a graded canonical module $\omega_S$, then one has  
\[
\omega_R \cong \sExt^{n - d}_S (R, \omega_S).
\]
\end{thm}
\begin{proof}
By \Cref{sCompletion}(b) it suffices to show the isomorphism after $^*$-completions. Since $\phi(S_0) = R_0$, by $^*$-completing both $R$ and $S$ as $S$-modules we may assume that $R$ and $S$ are $^*$-complete. Let $\m_{S_0}$ and $\m_{R_0}$ denote the maximal ideals of $S_0$ and $R_0$, respectively. Write $E' := E_{S_0} ( S_0/ \m_{S_0})$ and $E :=  E_{R_0} ( R_0/ \m_{R_0})$.
By \cite[Exercise 13]{HuLLC} we have $E \cong \Hom_{S_0} (R_0, E')$. 
From the hom-tensor adjointness and the graded version of the local duality theorem \cite[Theorem 3.6.19(b)]{BH}, we obtain the following isomorphisms
\begin{align*}
\omega_R &\cong \sHom_{R_0} ( {\HH}^d_{\MM} (R), E ) \\
	&\cong \sHom_{R_0} ( {\HH}^d_{\MM} (R),  \Hom_{S_0} (R_0, E')) \\	
	&\cong \sHom_{S_0} ( {\HH}^d_{\MM} (R) \otimes_{R_0} R_0,  E') \\
	&\cong \sHom_{S_0} ( {\HH}^d_{\MN} (R),  E') \\	
	&\cong \sExt^{n - d}_S (R, \omega_S),
\end{align*} 
and this completes the proof.
\end{proof}

\begin{cor}\label{lcmExt} Let $R$ be a Noetherian $\ZZ$-graded ring with a unique maximal homogeneous ideal $\MM$. Let $S = A[X_1,\dots,X_n]$ be a $\ZZ$-graded polynomial ring over a Gorenstein local ring $A$. Assume that there exists a surjective graded ring homomorphism $\phi : S \to R$ with $\phi(A) = R_0$ and $\MM$ is maximal. Let $\MN := \phi^{-1}(\MM)$, $g = \he \ker (\phi)$, and $\rho = \sum_{i= 1}^n \deg (X_i)$. Then one has $\omega_{S_{(\MN)}} \cong {S_{(\MN)}}(- \rho)$ and 
\begin{align*}
\omega_R  &\cong \sExt^g_{S_{(\MN)}} ( R, S_{(\MN)}) (- \rho)\\
	&\cong \sExt^g_S ( R, S) (- \rho).
\end{align*}
\end{cor}

\begin{proof}
Since $ \sExt^g_S ( R, S)$ is a graded $R$-module and $R = R_{(\MM)}$, we have
\begin{align*}
 \sExt^g_S ( R, S)  &\cong  \sExt^g_S ( R, S)_{(\MM)}   \cong \sExt^g_S ( R, S)_{(\MN)}  \cong \sExt^g_{S_{(\MN)}} ( R_{(\MN)}, S_{(\MN)}) \\
	  	  &\cong \sExt^g_{S_{(\MN)}} ( R_{(\MM)}, S_{(\MN)}) 	 \cong \sExt^g_{S_{(\MN)}} ( R, S_{(\MN)}). 
\end{align*}
Therefore, it suffices to show the first isomorphism. \\

First we show that we can reduce to the case where all $X_i$ are in $\MN$. If $\deg (X_i) \neq 0$, then $\phi(X_i)$ is a homogeneous element of degree not equal to zero in $R$. Since $\MM$ is maximal, there is no homogeneous unit of degree not equal to zero. Hence $\phi(X_i) \in \MM$, i.e., $X_i \in \MN$. Therefore, if $X_i \notin \MN$, then $\deg (X_i) = 0$. Suppose $X_i \notin \MN$. 
Since $\phi (X_i) \in R_0 = \phi(A)$, there exist $z_i \in A \subseteq S_0$ such that $z_i = \phi(X_i)$. Since $X_i - z_i$ is in $\ker (\phi)$, $X_i - z_i \in \MN$. Replacing the variable $X_i$ by $X_i - z_i$, we may assume that $X_i \in \MN$. \\

Since $X_i \in \MN$ for all $i$, we have $\MN = (\m, X_1,\dots, X_n)$ where $\m$ is the maximal ideal of $A$. Let $S' = S_{(\MN)}$.  Since $\MN = \phi^{-1} (\MM)$, $\phi$ factors through $S'$. Write $\phi': S' \to R$ for the ring homomorphism induced by $\phi$. Since $S'$ and $R$ have unique maximal homogeneous ideals, $\phi'$ surjective, and $\phi'(A) = R_0$, by \Cref {GddHK5.12} we are done once we have shown that $\omega_{S'} \cong S' (-\rho)$. 
By \Cref{sCompletion}(b) it suffices to show the isomorphism after $^*$-completion. Hence we may assume that $S'$ is $^*$-complete. 
 Write $\textbf{\underline{X}} = X_1,\dots, X_n$. 
Then $\textbf{\underline{X}}$ form a homogeneous regular sequence on $S'$, and $S' / (\textbf{\underline{X}}) S' \cong S/(\textbf{\underline{X}})_{(\MN)} \cong A_{(\MN)} = A_{\m} = A$. 
Since $A$ is a Gorenstein local ring concentrated in degree zero, we have $\omega_A \cong A$. Then $S'/ (\textbf{\underline{X}})S' \cong A \cong \omega_A \cong \omega_{S'/ (\textbf{\underline{X}})S' }$ and  \Cref{ConvRegHomSeq} imply the isomorphism $\omega_{S'} \cong S' (-\rho)$. This completes the proof.
\end{proof}

Recall that a Noetherian $\ZZ$-graded ring $R$ with a unique maximal homogeneous ideal is called \ti{quasi-Gorenstein} if $\omega_R \cong R(a)$ for some $a \in \ZZ$. If the unique maximal ideal is maximal, the number $a$ is well-defined, and it is called the \textbf{a}-invariant of the ring $R$. 
A {Gorenstein} $\ZZ$-graded ring is a quasi-Gorenstein ring which is Cohen-Macaulay. The ring $S_{(\MN)}$ in the remark above is a Gorenstein ring. Notice that this definition agrees with \cite[Theroem 3.6.19]{BH} when $R$ is Cohen-Macaulay. However, we do not require a ring to be Cohen-Macaulay.

\section*{Graded canonical modules of extended Rees algebras}
\begin{prop}\label{prop:TqGorRGor}
Let $(R,\m)$ be a Cohen-Macaulay local ring having a canonical module $\omega_R$. Let $I \subseteq \m$ be an $R$-ideal. If $R[It,t^{-1}]$ is quasi-Gorenstein, then $R$ is Gorenstein. 
\begin{proof}
The canonical module of $R[t,t^{-1}]$ is isomorphic to $\omega_R [t,t^{-1}]$ \cite[Proposition 4.1]{Ao83}. 
Since $R[t,t^{-1}] = R[It,t^{-1}]_{(t^{-1})^{-1}}$ and canonical modules localize, $\omega_R [t,t^{-1}] \cong (\omega_{R[It,t^{-1}]})_{ {(t^{-1})}^{-1}} \cong (R[It,t^{-1}])_{ {(t^{-1})}^{-1}}(a) = R[t,t^{-1}](a)$ where $a$ is the \textbf{a}-invariant of $R[It,t^{-1}]$. Hence we have $R \cong \omega_R$, i.e., $R$ is Gorenstein. 
\end{proof}
\end{prop}

\begin{lem}\label{lem:canModofG}
Let $(R,\m)$ be a Cohen-Macaulay local ring having a canonical module $\omega_R$ and $I \subseteq \m$ an ideal. If $R[It,t^{-1}]$ is quasi-Gorenstein, then the canonical module $\omega_{\gr_I(R)}$ of $\gr_I(R)$ is of rank $1$, and  there exists an inclusion $\gr_I (R) (\rho) \subseteq \omega_{\gr_I (R)}$ of graded $\gr_I(R)$-modules where $\rho = a( R[It, t^{-1}]) -1$.
\begin{proof}
Write $T = R[It,t^{-1}]$ and $G = \gr_I(R)$. We need to show that $(\omega_G)_\p \cong G_\p$ for all the associated primes of $\omega_G$. Since $G$ is equidimensional and unmixed, $\Ass(\omega_G) = \Min(G)$ \cite[(1.7)]{Ao83}. 
Let $\pi: T \stackrel{\hbox{nat}} {\longrightarrow} G$ and $\p \in \Min(G)$. Write $P = \pi^{-1}(\p)$. Since $T$ satisfies Serre's condition $(S_2)$ and $P \in \Min (T/ (t^{-1}))$, $P$ is of height $1$. Hence $T_P$ is Gorenstein. Since $G_\p \cong T_P/(t^{-1})_P$, $G_\p$ is Gorenstein, and this shows that $(\omega_G)_\p \cong G_\p$ is of rank $1$.\\

Now, we show the inclusion $\gr_I (R) (\rho) \subseteq \omega_{\gr_I (R)}$. Let $S := R[X_1,\dots,X_n]$ be a $\ZZ$-graded polynomial ring over $R$ such that $\phi: S  \to R[It, t^{-1}]$ is a surjective homogeneous ring homomorphism. Let $H = \ker (\phi)$ and $g = \he H$. Observe that $G = S/(H,h)$ where $\phi(h) = t^{-1}$. From the exact sequence of graded $T$-modules $0 \rightarrow T(1) \stackrel{t^{-1}}{\longrightarrow} T \rightarrow G \rightarrow 0 $, 
we have an exact sequence of graded $\Ext$ modules
\[
0 \rightarrow \Ext^g_S( T, S)  \stackrel{t^{-1}}{\rightarrow} \Ext^g_S( T(1), S) \rightarrow \Ext^{g+1}_S(G,S).
\]
By \Cref{lcmExt} this shows that 
\begin{equation}\label{LCT} 
0 \rightarrow \omega_T \stackrel{t^{-1}}{\rightarrow} \omega_T (-1) \rightarrow \omega_G,
\end{equation}
is exact and the the result follows since $\omega_T \cong T(a)$, where $a$ is the \textbf{a}-invariant of $T$. 
\end{proof}
\end{lem}

\begin{thm}\label{tgrelation}
Let $(R,\m)$ be a Cohen-Macaulay local ring having a canonical module $\omega_R$ and $I \subseteq \m$ an ideal.  Assume that a graded canonical module $\omega_{R[It, t^{-1}]}$ of $R[It, t^{-1}]$ satisfies Serre's condition $(S_3)$. We have the following statements. 
\begin{enumerate}[$(a)$]
\item If $R[It,t^{-1}]$ is quasi-Gorenstein, then $(\omega_{R[It,t^{-1}]} / t^{-1} \omega_{R[It,t^{-1}]}) (1)$ is graded canonical module of $\gr_I (R)$.
\item $R[It,t^{-1}]$ is quasi-Gorenstein if and only if $\gr_I(R)$ is quasi-Gorenstein. 
\end{enumerate}
\begin{proof}
Let $G = \gr_I (R)$ and $T = R[It,t^{-1}]$. Let $\omega_G$ be the graded canonical module of $G$. \\
\noindent (a): We have
\[
0 \rightarrow (\omega_T / t^{-1} \omega_T) (-1) \rightarrow \omega_G \rightarrow L \rightarrow 0
\]
where $L$ is the cokernel of the natural map in \eqref{LCT}. We show that $L = 0$. It is equivalent to showing that $L_\p = 0$ for all $\p \in \Ass(L)$. 
Since $\omega_T$ satisfies Serre's condition $(S_3)$, $\omega_T / t^{-1} \omega_T$ satisfies Serre's condition $(S_2)$ as a $G$-module. Let $\p \in \Ass(L)$. Since $T_P$ is Gorenstein for all prime ideals of height at most $2$, $G_\p$ is Gorenstein for all prime ideals of height at most $1$. Hence if $\p \in \Ass(L)$, then $\he \p \ge 2$. Since $\p$ is an associated prime of $L$, $\depth L_\p = 0$. However, $\depth (\omega_T / t^{-1} \omega_T)_\p \ge 2$ and $\depth (\omega_G)_\p \ge 2$. This implies that $\depth L_\p \ge 1$, and this is a contradiction. \\

\noindent$(b)$: The forward direction follows immediately from part $(a)$. For the other direction, we only need to show that $\omega_T$ is a cyclic faithful $T$-module. 
From the isomorphism $\omega_G \cong \omega_T /t^{-1} \omega_T$ and $\mu_T (\omega_T) = \lambda_T (\omega_T / \MM_T \omega_T)= \mu_G (\omega_T/ t^{-1} \omega_T) = 1$, where $\MM_T$ is the maximal homogeneous ideal of $T$, we conclude that $\omega_T$ is cyclic. 
Here, $\lambda(-)$ denotes the length of a module.
Also, $\omega_T$ is faithful since $T$ is unmixed cf. \cite[(1.8)(a)(c)]{Ao83}. Therefore  $\omega_T \cong T(a(G) + 1)$.
\end{proof}
\end{thm}

\section*{The Main theorems}
It is interesting to see under which conditions quasi-Gorenstein extended Rees algebras are Gorenstein. To this end, Heinzer, M.-K.\ Kim, and Ulrich posed the next question. In this section we present two cases which give an affirmative answer to the question.

\begin{ques}[{\cite[Question 4.11]{HKU05}}]\label{Q:HKU05}
Let $(R,\m)$ be a local Gorenstein ring. Let $I$ be an $\m$-primary ideal. If the extended Rees algebra $R[It,t^{-1}]$ is quasi-Gorenstein, then is it Gorenstein? 
\end{ques}

In the same paper, the authors characterized the quasi-Gorenstein property of the extended Rees algebra in terms of colon ideals. 

\begin{prop}[{\cite[Theorem 4.1]{HKU05}}]\label{HKU4.1} Let $(R,\m)$ be a Gorenstein local ring of dimension $d$, and let $I$ be an $\m$-primary ideal. Assume that $J \subseteq I$ is a reduction of $I$ with $\mu(J) = d$. Let $r:=r_J(I)$ be  the reduction number of $I$ with respect to $J$, and let $k$ be an integer such that $k \ge r$. Then the graded canonical module $\omega_{R[It, t^{-1}]}$ of $R[It, t^{-1}]$ has the form 
\[
\omega_{R[It, t^{-1}]} \cong \bigoplus_{i \in \ZZ}(J^{i+k}:_R I^k)t^{i+d-1}.
\]
In particular, for $a \in \ZZ$,  the following are equivalent$:$
\begin{enumerate}[$(a)$]
\itemsep0em 
\item $R[It, t^{-1}]$ is quasi-Gorenstein with \textbf{a}-invariant $a$.
\item $J^i:_R I^r = I^{i+a-(r-d+1)}$ for every $i \in \ZZ$.
\end{enumerate}
\end{prop}

\begin{dfn} Let $R$ be a Noetherian local ring.
 Let $I$ be an ideal and $J$ a minimal reduction of $I$. Then the \textit{index of nilpotency}, denoted by $s_J(I)$, is $\min \{ i  \mid I^{i+1} \subset J \}$, and $s(I) = \max \{ s_J(I) \mid J \hbox{ is a minimal reduction of I } \}$.  
\end{dfn}

\begin{lem}[{\cite[Remark 4.4]{HKU05}}]\label{HKU4.4}
We use the setting of \Cref{HKU4.1}. In addition, assume that $R[It,t^{-1}]$ is quasi-Gorenstein. Then one has 
\begin{enumerate}[$(a)$]
\item $s_J(I) - d + 1 \le a(R[It,t^{-1}]) \le r_J(I) - d  + 1$ and 
\item $\max\{ \; n \mid  I^{r_J(I)} \subseteq J^n \} = r_J(I) - d + 1 - a(R[It,t^{-1}])$.
\end{enumerate}
\end{lem}

Now, we are ready to provide a positive answer to \Cref{Q:HKU05} for a class of monomial ideals in a polynomial ring. 
\begin{thm}\label{thm:qGor1}
Let $R$ be a polynomial ring in $d$-variables over a field and $I$ a monomial ideal of height $d$. Assume that $I$ has a $d$-generated monomial reduction. If $R[It, t^{-1}]$ is quasi-Gorenstein, then $R[It,t^{-1}]$ is Gorenstein.
\begin{proof}
We show that $\gr_I (R)$ is Cohen-Macaulay. Let $J  = (g_1, \dots, g_d)$ be the monomial reduction of $I$, and write $r := r_J(I)$. Let $u = r - d + 1 -a(R[It,t^{-1}])$ where $a(R[It,t^{-1}])$ is the \textbf{a}-invariant of $R[It,t^{-1}]$. Since $R[It,t^{-1}]$ is quasi-Gorenstein, we have $J^i :I^r = I^{i-u}$ for all $i$ by \Cref{HKU4.1}(b). The Cohen-Macaulayness of $\gr_I (R)$ follows by the Valabrega-Valla criterion \cite[Theorem 1.1]{VV78} once we have shown that $J \cap I^i \subseteq J I^{i-1}$ for $0 \le i \le r$. Recall that all ideals in question are monomial ideals. Let $a$ be an arbitrary monomial in $J \cap I^i$. Since $a \in J$ and $J$ is a monomial ideal generated by the $g_i$'s, we can write $a = a'g$ where $a' \in R$ and $g = g_j$ for some $j$. 
We want to show that $a' \in I^{i-1}$. 
For an arbitrary element $z \in I^r$, we have $az \in I^{i+r} = J^i I^r \subseteq J^i J^u$ where the last inclusion follows from \Cref{HKU4.4}(b). Then $az = a'zg \in J^{i+u}$, and this implies $a'z \in J^{i +u -1}$ since $g + J^2$ is a non zero-divisor on $\gr_J(R)$. Our choice of $z \in I^i$ was arbitrary. Hence we conclude that $a' \in J^{i + u -1}:I^r = I^{i-1}$. 
\end{proof}
\end{thm}

In the rest of this section we use the setting of \Cref{HKU4.1} and study quasi-Gorenstein extended Rees algebras under the condition that $s_J(I) = r_J(I)$ for some $d$-generated minimal reduction $J$ of $I$. This is a necessary condition if the associated graded ring $\gr_I(R)$ is Cohen-Macaulay by \cite[Theorem 1.1]{VV78}.

\begin{rmk}\label{simpleU}
Let $(R,\m)$ be a $d$-dimensional Gorenstein local ring and $I$ an $\m$-primary ideal. Assume that  $s_J(I) = r_J(I)$ for some $d$-generated minimal reduction $J$ of $I$. Then $R[It,t^{-1}]$ is quasi-Gorenstein if and only if $J^i : I^r = I^i$ for all $i \in \mathbb{Z}$.
\end{rmk}

\begin{proof}
This follows directly from \Cref{HKU4.1} and \Cref{HKU4.4}(a). 
\end{proof}

\begin{dfn}[{\cite{Ro00}}] 
Let $(R,\m)$ be a Noetherian local ring and $I$ an $\m$-primary ideal. The ideal $I$ called \textit{$n$-standard} if $J \cap I^i = J I^{i-1}$ for all $i \le n$.  
\end{dfn}

\begin{prop}
Let $(R,\m)$ be a $d$-dimensional Cohen-Macaulay local ring having a canonical module and $I$ an $\m$-primary ideal. Assume that $R[It,t^{-1}]$ is quasi-Gorenstein and $s_J(I) = r_J(I)$ for some $d$-generated minimal reduction $J$ of $I$. Then $I$ is 2-standard.
\begin{proof}
Since $R[It,t^{-1}]$ is quasi-Gorenstein and $s_J(I) = r_J(I)$, we have $J^i : I^r = I^{i}$ for all $i \in \mathbb{Z}$ by \Cref{simpleU}. We show that $J \cap I^2 \subseteq JI$. 
Write $J = (x_1,\dots,x_d)$, and let $a \in J \cap I^2$. Then we may write $a = \sum a_i x_i$ for some $a_i$ in $R$. For any $z \in I^r$, $a z \in I^r I^2 = I^{r+2} = J^2I^r \subseteq J^2$, that is $z \sum a_i x_i = \sum z a_i x_i \in J^2$. Since $J/J^2$ is a free $R/J$-module with basis $x_1 + J^2, \dots, x_d + J^2$, we obtain $z a_i \in J$. This implies that indeed $a_i \in J:I^r = I$.
\end{proof}
\end{prop}

\begin{lem}\label{a-inv2}
Let $(R,\m,\kk)$ be a $d$-dimensional Cohen-Macaulay local ring having a canonical module with infinite residue field $\kk$. Let $I$ be an $\m$-primary ideal. Assume that $R[It,t^{-1}]$ is quasi-Gorenstein and $s_J(I) = r_J(I)$ for some minimal reduction $J$ of $I$. 
\begin{enumerate}[$(a)$]
\item One has $r_J(I) = r(I)$ and $s_J(I) = s(I)$; in particular, $s(I) =  r(I)$. 
\item $a(\gr_I (R))  = a(R[It,t^{-1}]) - 1 = d - r_J (I)$ where $a(\gr_I(R))$ and $a(R[It,t^{-1}])$ are the \textbf{a}-invariants of $\gr_I (R)$ and $R[It,t^{-1}]$, respectively. 
\end{enumerate}
\begin{proof}
(a): By \Cref{HKU4.4}(a) we have $s_J(I) - d + 1 = a := a(R[It,t^{-1}]) = r_J(I) - d + 1$. Let $L, K \subseteq I$ be minimal reductions with $r_{L} (I) = r (I)$ and $s_{K} (I) = s(I)$, respectively. Apply \Cref{HKU4.4}(a) to see that $s (I) -d + 1 \le a \le r(I) - d + 1$. This shows that $s(I) \le s_J(I)$ and $r(I) \ge r_J(I)$. The other direction of the inequalities follows from the definition of $s (I)$ and $r(I)$. \\

\noindent (b): From the exact sequence $\eqref{LCT}$ in the proof of \Cref{lem:canModofG}, we have $a(\gr_I(R)) \ge a(R[It,t^{-1}]) - 1 = r -d$, and the other inequality follows from a result of Trung \cite[Proposition 3.2]{Tr87}.
\end{proof}
\end{lem}

\begin{thm}\label{thm:qGor2}
Let $(R,\m)$ be a Cohen-Macaulay local ring having a canonical module with infinite residue field. Let $I$ be an $\m$-primary ideal. Assume that $I$ is an almost complete intersection ideal, i.e., $\mu (I) \le \he I  + 1$. The following are equivalent:
\begin{enumerate}[$(a)$]
\item $R[It,t^{-1}]$ is Gorenstein. 
\item $R[It,t^{-1}]$ is quasi-Gorenstein and $s_J(I) = r_J(I)$ for some minimal reduction $J$ of $I$.
\end{enumerate}
\begin{proof}
Recall that if $I$ is a complete intersection, i.e., $\mu (I) = \he I$, then $R[It,t^{-1}]$ is a Cohen-Macaulay ring, and the equivalence follows immediately. Suppose that $\mu (I) = \he I + 1$. The implication $(a) \Longrightarrow (b)$ is obvious. Let $d = \dim R$ and $J$ be a minimal reduction such that $r = r_J(I) = s_J(I)$. Choose a generating set $x_1, \dots, x_d$ of $J$. Since $J$ is a minimal reduction of $I$, a generating set of $J$ can be extended to that of $I$. Hence we may write $I = J + (x)$ for some $x$ in $R$. By \cite[Theorem 1.1]{VV78} it suffices to show that $J \cap I^{i} \subseteq JI^{i-1}$ for $1 \le i \le r$.  Since $I = J + (x)$, 
\begin{align*}
J \cap I^{i} \subseteq JI^{i-1} &\Longleftrightarrow J  \cap (J+(x))^i \subseteq JI^{i-1} \\
	&\Longleftrightarrow J \cap (J^i + xJ^{i-1} + \cdots + (x)^i )\subseteq JI^{i-1} \\
	&\stackrel{(\star)}\Longleftrightarrow J^i + xJ^{i-1} + \cdots  + x^{i-1}J +  J \cap (x)^i \subseteq JI^{i-1}\\ 
	&\Longleftrightarrow J \cap (x)^i \subseteq JI^{i-1},
\end{align*}
where the equivalence ($\star$) follows from the containment $J^i, \dots, x^{i-1} J \subseteq J $. 
First, we claim that $J \cap (x)^i \subseteq I^{i+1}$. 
Observe that $J \cap (x)^i = J \cap (x^i) = x^i (J :_R x^i) = x^i ( J:_R I^i)$. One has $I^r  x^i(J: I^i) \subseteq I^r I^i (J:I^i) \subseteq I^{r+i} (J : I^i) = J^i I^r (J : I^i) = J^i I^{r-i} I^i (J : I^i) \subseteq J^i I^{r-i} J \subseteq J^{i+1}$. Therefore $x^i ( J: I^i) \subseteq J^{i+1} : I^r = I^{i+1}$, and the shows the claim. 
Now, we apply decreasing induction on $i$. When $i = r$, $J \cap (x)^r \subseteq I^{r+1} = J I^r \subseteq J I^{r-1}$. For $i < r$, $J \cap I^{i+1} = J I^i$ by the induction hypothesis. Hence we have $J \cap (x)^i \subseteq J \cap I^{i+1} = JI^i \subseteq J I^{i-1}$.
\end{proof}
\end{thm}

\section*{Results on the \textbf{a}-invariant and the core of an ideal}

Let $R$ be a Noetherian local ring and $I$ an $R$-ideal. 
In this section, we present results on the core of powers of an ideal $I$ and the \textbf{a}-invariant of the extended Rees algebras when the extended Rees algebra $R[It,t^{-1}]$ is quasi-Gorenstein.

\begin{thm}\label{thm:core}
 Let $(R,\m)$ be a Cohen-Macaulay local ring having a canonical module and $\kk = R/ \m$. Let $I$ be an $\m$-primary ideal. Assume that $R[It,t^{-1}]$ is quasi-Gorenstein and either characteristic of $\kk$ is zero or greater than $r(I)$. Let $a := a(R[It,t^{-1}])$ be the \textbf{a}-invariant of $R[It,t^{-1}]$. Then $\core(I^u) = I^{du + a}$ for all $u \in \mathbb{Z}$. 
\begin{proof}
We may assume that the residue field is infinite. Let $J$ be a minimal reduction of $I$ with $r := r_J(I) = r(I)$. Fix a minimal generating set $x_1, \dots, x_d$ of $J$ where $d = \dim R$. Let $J' = J^{[u]} := ( x_1^u, \dots, x_d^u )$ and $I' = I^u$. Then $J'$ is a minimal reduction of $I'$. By \cite[Theorem 4.5]{PU05} $\core (I) = J^{n+1}:I^n$ for $n \ge r(I)$.  We compute the core of $I'$. We use \cite[Lemma 2.2]{PUV07} which shows $J^{n + 1}: I^n = J^{[n]} : I^{dn}$ for $n  \gg 0$ in our setting. 
For $n \gg 0$, one has
\begin{align*}  
\core (I') &= J'^{n+1} : I'^n \\
  &= J'^{[n+1]} : I'^{dn}  \\
  &= (J^{[u]})^{[n+1]} : (I^u)^{dn} \\
  &= J^{[nu + u]} : I^{udn}\\
  &= J^{[nu + u]} : I^{udn - r}I^r \\
  &= J^{[nu + u]} : J^{udn - r}I^r \\
  &= (J^{[nu + u]} : J^{udn - r}) :I^r\\
  &= J^{d(nu+u - 1) + 1 - (udn -r) } :I^r \\
  &= J^{ du -d +1 + r} : I^r \\
  &= I^{du + a},
\end{align*}
where the last equality follows from \Cref{HKU4.1}(b).
\end{proof}
\end{thm}

\begin{dfn}
Let $S$ be a $\mathbb{Z}$-graded ring. Let $M$ be a graded $S$-module. The initial degree of $M$, denoted by $\indeg_S (M)$, is the $\inf \{ i \in \mathbb{Z} \mid [M]_i \neq 0 \; \}$ if $M \neq 0$, and is $0$ if $M= 0$. 
\end{dfn}

\begin{rmk}
The number $\indeg_S (M)$ can be $- \infty$ in general. 
For a finitely generated $\mathbb{Z}$-graded module $M$ over a Noetherian ring having a unique maximal homogeneous ideal $\MM$, which is maximal, $\indeg_S ( M / \MM M)$ is a well-defined finite number. In this case, $\indeg_S ( M / \MM M)$ is the minimum among the degrees of the elements in a minimal homogeneous generating set of $M$.
\end{rmk}

 \begin{lem}\label{q-shift}
Let $(R,\m)$ be an analytically unramified Cohen-Macaulay local ring having a canonical module with infinite residue field. Let $I$ be an $\m$-primary ideal. Write $T := R[It,t^{-1}]$, and let $\MM_T$ be the maximal homogeneous ideal of $T$. Let $\overline{T}$ be the integral closure of $T$ in $R[t,t^{-1}]$. Let $\mathcal{C}$ denote the conductor ideal, i.e., $\mathcal{C} = T :_T \overline{T}$. 
If $T$ is quasi-Gorenstein, then $a(\overline{T} ) = a(T) - \indeg ( \mathcal{C} / \MM_T \mathcal{C} )$ where $a(-)$ denotes the \textbf{a}-invariant. 
Furthermore, we have $\indeg ( \mathcal{C} / \MM_T \mathcal{C} ) \le 0$ and equality happens if and only if $T$ is integrally closed in $R[t, t^{-1}]$. 

 \begin{proof}
 If $T$ is integrally closed, then there is nothing to prove. We assume that $T$ is not integrally closed. Let $d = \dim R$ and $A := R[Jt, t^{-1}]$ where $J$ is a minimal reduction of $I$. Then $A$ is a Cohen-Macaulay ring with $A_0 = [\overline{T}]_0  = R$, and  since $R$ is analytically unramified, $A \subseteq \overline{T}$ is module-finite extension. 
 Hence the graded canonical modules $\omega_T$ of $T$ and $\omega_{\overline{T}}$ of $\overline{T}$ are $\Hom_A ( T, \omega_A)$ and $\Hom_A ( \overline{T}, \omega_A)$, respectively, where $\omega_A$ is the graded canonical module of $A$. We claim that $\omega_{\overline{T}} \cong \Hom_T (\overline{T}, \omega_T)$. 
Observe that $\Hom_T (\overline{T}, \omega_T) \cong \Hom_T (\overline{T}, \Hom_A ( T, \omega_A ) ) \cong \Hom_A (\overline{T} \otimes_T T , \omega_A ) \cong \Hom_A (\overline{T}, \omega_A)$, where the last module is a graded canonical module of $\overline{T}$, see \Cref{GddHK5.12}. Since $\overline{T}$ has a unique maximal homogeneous ideal which is maximal, it has well-defined \textbf{a}-invariant  $- \indeg (\omega_{\overline{T}} / \MM_{\overline{T}} \omega_{\overline{T}})$.\\
 
Since $T \subset \overline{T}$ is birational, i.e., they have the same total quotient ring, we have $\Hom_T(\overline{T}, T)  \cong  ( T :_T  \overline{T} )$ where the last module is the conductor ideal $\mathcal{C}$. Recall that since $T$ is quasi-Gorenstein, $\omega_T \cong T(a(T))$. Therefore $\omega_{\overline{T}} \cong \mathcal{C}(a(T))$. This implies that $a(\overline{T}) = - \indeg ( \mathcal{C}/ \MM_{\overline{T}} \mathcal{C} (a(T))  )  =  - (\indeg (\mathcal{C}/ \MM_{\overline{T}} \mathcal{C}) - a(T))$. \\

It remains to show that $\indeg (\mathcal{C}/ \MM_{\overline{T}} \mathcal{C}) = \indeg (\mathcal{C}/ \MM_{T} \mathcal{C})$.
Since $R$ is analytically unramified, there exists a positive integer $q$ such that $\overline{I^i} = I^{i-q} \overline{I^q}$ for $i \ge q \ge 0$ \cite[Theorem 1.4]{Re61}. This shows that $t^{-q} \in \mathcal{C}$. 
By choosing $q$ the smallest positive integer with the property, we have $\indeg_T (C / \MM_T C) = -q \le 0 $. 
Since $[T]_i = [\overline{T}]_i$ for all $i \le 0$, in particular, $[\MM_T]_i = [\MM_{\overline{T}}]_i$ for $i < 0$, we obtain $\indeg (\mathcal{C}/ \MM_{\overline{T}} \mathcal{C}) = \indeg (\mathcal{C}/ \MM_{T} \mathcal{C})$.
 \end{proof}
 \end{lem}

 \begin{thm}\label{thm:a-inv} 
Let $(R,\m)$ be a $d$-dimensional analytically unramified Cohen-Macaulay local ring having a canonical module with infinite residue field. Let $I$ be an $\m$-primary ideal. Let $\{\mathcal{F}_i\}_{i \in \mathbb{Z}}$ where $\mathscr{F}_i = \overline{I^i}$ be the integral closure filtration where $\mathcal{F}_i = R$ when $i \le 0$. Assume that $\overline{T} = \oplus_{i \in \mathbb{Z}} \mathcal{F}_i t^i$ is Cohen-Macaulay. If $R[It,t^{-1}]$ is quasi-Gorenstein, then the index of nilpotency does not depend on a minimal reduction of $I$ and the \textbf{a}-invariant of $R[It,t^{-1}]$ is $s(I) - d + 1$.
 \begin{proof} 
Let $T= R[It,t^{-1}]$, $J$ be a minimal reduction of $I$, and $q = a(\overline{T}) - a(T)$ where $a(-)$ denotes the \textbf{a}-invariant. 
First, we claim that $s_J(I) = s_J(\mathcal{F}) - q$ where $s_J(\mathcal{F}) := \min \{ i \mid \mathcal{F}_{i+1} \subseteq J \}$. 
Since $\overline{T}$ is Cohen-Macaulay, $a(\overline{T}) = s_J(\mathcal{F}) - d + 1$, equivalently $a(T) =  s_J(\mathcal{F}) - q - d + 1$.
\Cref{HKU4.4}(a) $s_J(I) - d + 1 \le a(T)$ implies that $s_J(I) + q \le s_J(\mathcal{F})$. Since $t^{-q} \in T :_T \overline{T}$, we have $\overline{I^{n + q}} \subset I^n$. Hence we have ${J : \overline{I^{n + q}}} \supseteq {J : I^n}$, and by definition this implies the other inequality $s_J(I) + q \ge s_J(\mathcal{F})$. This proves the claim. 
Since $a(T) = s_J(I) - d + 1$ is independent of a minimal reduction $J$, we have $s_J(I) = s(I)$. This completes the proof. 
 \end{proof}
 \end{thm}

The ring $\oplus_{i \in \mathbb{Z}} \mathcal{F}_i t^i$ is not Cohen-Macaulay in general. However, Hochster \cite{Ho72} showed that it is Cohen-Macaulay when $R$ is a polynomial ring (localized at the origin) over a field and $\mathcal{F}_1$ is a monomial ideal.

\section{The Gorensteinness of extended Rees algebras of monomial normal ideals}

Let $I$ be a monomial ideal in a polynomial ring $R = k[x_1,\dots, x_d]$ over a field $k$. The integral closure of the extended Rees algebra $\overline{R[It, t^{-1}]}$ is Cohen-Macaulay by a result of Hochster \cite{Ho72}.  
In \cite[Theorem 5.6]{HKU11} the authors characterized the Gorensteinness of $\overline{R[It, t^{-1}]}$ when $I$ has a minimal reduction $J$ which is generated by powers of variables, i.e.,  $J = (x_1^{a_1},\dots, x_d^{a_d})$ for some $a_i \in \mathbb{N}$. This condition having such a minimal reduction is equivalent to the condition that $I$ has only one Rees valuation. In this section we generalize this result by removing the condition on the number of Rees valuations. We are able to interpret the reduction number that appears in \cite[Proposition 5.4]{HKU11} in terms of the \textbf{a}-invariant of $\overline{R[It,t^{-1}]}$, and this leads to a lower bound on the reduction number. We follow the notation of Chapter $6$ of \cite{BH}. 

\begin{set}\label{set:nor} 
 Let $R = k[x_1,\dots, x_d]$ be a polynomial ring in $d$-variables over a field $k$ and $\m = (x_1,\dots,x_d)R$. 
 We assign a $\mathbb{Z}^{d+1}$-grading to the Laurent polynomial ring $R[t, t^{-1}]$ by setting the \textit{exponent function}, denoted by $\exp$, from the set of monomial ideals of $R$ to $\mathbb{Z}^d$ as $\exp (x_1^{a_1}\cdots x_d^{a_d} t^{a_{d+1}}) = (a_1,\dots,a_d, a_{d+1})$. This determines the grading, since $\{ \exp(x_1), \dots \exp (x_d), \exp(t) \}$ forms an (orthonormal) $\mathbb{Z}$-basis for $\mathbb{Z}^{d+1}$. Let $A$ be a $\mathbb{Z}^{d+1}$-graded subring of $R[t,t^{-1}]$. 
 Then the semigroup 
 \[
{{C}}_A :=  \{ \exp (m) \mid m~\text{monomial of}~R[t,t^{-1}]~\text{in}~A \} \subseteq \mathbb{Z}^{d+1}
 \]
is called the $\textit{the affine semigroup of $A$}$ if ${{C}}_A$ is a finitely generated semigroup. For an affine semigroup ${{C}}$, let $\relint ({{C}})$ denote \textit{the relative interior} of ${{C}}$ that is $\displaystyle \relint ({{C}}) := {{C}} \cap \relint \mathbb{R}_{\ge 0} {{C}}$.  
\end{set}

\begin{lem}[{\cite[Proposition6.1.5]{BH}}]\label{Embedding} With \Cref{set:nor}, let ${{C}}$ be the affine semigroup of $R[\m t, t^{-1}]$. Let $W =  \mathbb{Z}^{d+1}$ and $\{ \ee_i \}$ be the standard basis of $W$. Let $\phi \in \operatorname{Aut}_{\mathbb{Z}} (W)$ be an automorphism of $W$ defined as follows; $\phi(\ee_i) = \ee_i + \ee_{d+1}$ for $i = 0, \dots, d$ and $\phi(\ee_{d+1}) = - \ee_{d+1}$. Then $\phi |_{{C}}$ is an embedding of ${{C}}$ into $\mathbb{Z}_{\ge 0}^{d+1}$.
\begin{proof}
One can easily check that $\phi$ is an automorphism on $W$. 
We show that $\phi |_{{C}}$ is an embedding. Since ${{C}}$ does not have any inverse (in the sense of affine semigroups), $\ker (\phi |_{{C}}) = 0$. It remains to show that $\phi ({{C}}) \subseteq \mathbb{Z}^{d+1}$. Since $\phi(\exp (x_i)) = (0,\dots,0,1,0,\dots,0,1)$ and $\phi(\exp (t^{-1}) ) = (0,\dots,0,1)$, it suffices to show that for $a_i \in \mathbb{Z}_{\ge 0}$ and $b \in \mathbb{Z}$ such that  $\sum_{i = 0}^d a_i \ge b$, the image  
$
\phi(\exp (x_1^{a_1} \cdots x_d^{a_d} t^b) ) 
$
is in $\mathbb{Z}_{\ge 0 }^{d+1}$.
Indeed we have 
\begin{align*}
\phi(\exp (x_1^{a_1} \cdots x_d^{a_d} t^b)) 
	&= \left( \sum_{i = 0}^d a_i \phi(\exp (x_i) ) \right) + b \phi(\exp (t) )\\ 
	&= \left( \sum_{i = 0}^d a_i \phi(\exp (x_i) ) \right) - b \phi(\exp (t^{-1}))\\
	&= (a_1,\dots,a_d, \sum_{i = 0}^d a_i) - b(0,\dots,0,1)\\ 
	&= (a_1,\dots,a_d, \sum_{i = 0}^d a_i - b) \in \mathbb{Z}_{\ge 0 }^{d+1}. \qedhere
\end{align*}
\end{proof}
\end{lem}

\begin{cor} With \Cref{set:nor}, let $\mathscr{F} = \{ \mathscr{F}_i \}_{i \in \mathbb{Z}}$ be a filtration where $\mathscr{F}_i = R$ when $i \le 0$ and $\mathscr{F}_i$  are monomial ideals contained in $\m^i$. Then the affine semigroup of $\oplus_{i \in \mathbb{Z}} \mathscr{F}_i t^i$ can be embedded into $\mathbb{Z}_{\ge 0}^{d+1}$. In particular, the  affine semigroup of $\overline{R[It,t^{-1}]}$ can be embedded into $\mathbb{Z}_{\ge 0}^{d+1}$ when $I$ is a monomial ideal.
\end{cor}

For a monomial ideal $I$ in $R$, its integral closure $\overline{I}$ can be determined by the Newton polyhedron of the ideal $I$. Here the \textit{Newton polyhedron} $\NP(I)$ of a monomial ideal $I$ is the cone generated (over $\mathbb{R}_{\ge 0}$) by the exponent vectors of the monomials in $I$ in $\mathbb{R}^d$. 
We would like to describe the ring $\overline{R[It,t^{-1}]}$ using the half spaces  in $\mathbb{R}^{d+1}$ that corresponds to the ones that determine $\overline{I}$ in $\mathbb{R}^d$. Let $\langle \quad, \quad \rangle$ denote the inner product in $\RR^n$.

\begin{lem}\label{Halfspaces} 
With \Cref{set:nor}, let $I$ be a monomial ideal in $R$. Let $H_i^+ = \{ v \in \mathbb{R}^d \mid \langle (a_{i1},\dots,a_{id}),v \rangle \ge h_i \}$ where $h_i \in \mathbb{Z}_{\ge 0}$ be the half spaces in $\mathbb{Z}^d$ that determine the Newton polyhedron of $I$. 
Define the half spaces $\widetilde{H}_i^+$ that correspond to each $H_i^+$ in $\mathbb{R}^{d+1}$ as $\widetilde{H}_i^+ := \{ v \in \mathbb{R}^{d+1} \mid \langle (a_{i1} - h_i, \dots, a_{id} - h_i, h_i), v \rangle \ge 0 \}$. 
Let ${{C}}$ be the affine semigroup of $\overline{R[It, t^{-1}]}$. 
Then the intersection $\cap \; \widetilde{H}_i^+$ is the cone generated by the affine semigroup $\phi({{C}})$ where $\phi$ is the embedding in \Cref{Embedding}.
\begin{proof}
Let $H := H_i^+$ for some $i$ and write $H = \{ v \in \mathbb{R}^d \mid \langle (a_1,\dots,a_d),v \rangle \ge h \}$. 
An exponent vector $(z_1, \dots, z_{d+1})$ in $\phi({{C}})$ is the image of $(z_1, \dots, z_d, -z_{d+1} + (z_1+ \cdots + z_d) )$ under the map $\phi$.  One has $(z_1, \dots, z_d, -z_{d+1} + (z_1+ \cdots + z_d)) \in {{C}}$ if and only if the monomial $x^{z_1} \cdots x^{z_d}$ is in $\overline{I^{-z_{d+1} + (z_1+ \cdots + z_d)}}$. 
In terms of half spaces this corresponds to the condition $\langle (a_1,\dots, a_d) , (z_1,\dots, z_d) \rangle \ge (- z_{d+1} + (z_1+ \cdots + z_d)) h$ for all the half space $H_i$ defining $\NP(I)$.
By setting $\widetilde{H} = \{ v \in \mathbb{R}^{d+1} \mid \langle (a_1- h, \dots, a_d - h, +h),v \rangle \ge 0 \}$, we obtain $\cap \; \widetilde{H}_i^+ = \mathbb{R}_{ \ge 0} \phi({{C}})$.
\end{proof}
\end{lem}

\begin{exam}
Let $R = \mathbb{C}[x_1, \dots, x_d]$ and $I = (x_1,\dots,x_d)$. Let $\{ \ee_i \}_{i = 1}^{d+1}$ be the standard base of $\mathbb{Z}^{d+1}$. Then one can easily see that $\phi({{C}}) = \{ (z_1,\dots,z_{d+1}) \in \mathbb{Z}^{d+1} \mid  z_i \ge 0 \hbox{ for all } i \} = \cap \{  (z_1,\dots,z_{d+1}) \in \mathbb{Z}^{d+1} \mid \langle \ee_i , v \rangle \ge 0 \}$ where $\phi$ as in \Cref{Embedding}. 
\end{exam}

\begin{lem}\label{PartCan}
With \Cref{set:nor} assume that $I$ is an $\m$-primary monomial ideal. Let ${{C}}$ be the affine semigroup of $\overline{R[It,t^{-1}]}$ and $\phi$ the embedding in \Cref{Embedding}. Then there exists an exponent vector of the form $(1,\dots,1,q)$ in $\phi({{C}})$ for some integer $q$, with $1 \le q \le d + 1$, which is part of a minimal generating set for the canonical ideal for $\overline{R[It,t^{-1}]}$.
\begin{proof}
Since $I$ is $\m$-primary, the half spaces of the form $\{ v \in \mathbb{R}^d \mid (0,\dots,1,\dots,0) \cdot v \ge 0 \}$, where the $i$th entry is $1$, are part of the boundary of the Newton polyhedron of $I$. By \Cref{Halfspaces} these will be part of the boundary half spaces in $\phi( {{C}})$. For instance, the half space $\{ v \in \mathbb{R}^d \mid (1,\dots,0) \cdot v \ge 0 \}$ corresponds to the half space $\{ v \in \mathbb{R}^{d+1} \mid (1,\dots,0) \cdot v \ge 0 \}$. Hence if $(z_1,\dots,z_d,z_{d+1}) \in \relint ({{C}})$, then $z_i \ge 1$ for $i = 0,\dots,d$. \\

By \cite[Theorem 6.3.5.(b)]{BH} it suffices to show that there exists an integer $q$ in $1 \le q \le d+1$ such that if $(1, \dots, 1, q)$ is in $\relint \phi({{C}})$, then $(1,\dots, 1, q - 1)$ is not in $\phi({{C}})$. Since $x_1 \cdots x_d \in R \subseteq \overline{R[It,t^{-1}]}$, $\phi(1,\dots,1,0) = (1, \dots,1, d)$ is in $\phi({{C}})$. 
If $(1, \dots, 1, d)$ is on the boundary, then $(1, \dots, 1, d+1)$ in $\relint \phi({{C}})$. In this case we set $q = d+1$. If $(1, \dots, 1, d)$ is not on the boundary, we can choose $q \le d$ be the minimal in the last component since we have $\phi(t^{-1}) = (0, \dots,0,1)$ in $\phi({{C}})$. 
\end{proof}
\end{lem}

\begin{thm}\label{normal:main} Let $R = k[x_1,\dots, x_d]$ be a polynomial ring in $d$-variables over a field $k$ and $\m = (x_1,\dots,x_d)R$. Let $I$ be an $\m$-primary monomial ideal and $H_i$ the half spaces that determine the Newton Polyhedron of $I$ where $H_i = \{ v \in \mathbb{R}^d \mid \langle (a_{i1},\dots,a_{id}),v \rangle \ge h_i \}$ for $h _i\in \mathbb{Z}_{\ge 0}$. Let $q$ as in \Cref{PartCan} and $w_i := \langle (a_{i1} - h_i, \dots, a_{id} - h_i, h_i) , (1, \dots, 1, q) \rangle$. Define $N^+_i := \{ v \in \mathbb{R}^{d+1} \mid \langle (a_{i1} - h_i, \dots, a_{id} - h_i, h_i), v \rangle \ge w_i \}$. Let ${{C}}$ be the affine semigroup of $\overline{R[It,t^{-1}]}$ and $\phi$ the embedding in \Cref{Embedding}.
Then $\overline{R[It, t^{-1}]}$ is Gorenstein if and only if  the relative interior of $\phi({{C}})$ is contained in $\cap \; N^+_i,$ equivalently $\displaystyle \relint (\phi({{C}})) = (\cap \; N^+_i) \cap \phi(C)$.
\begin{proof}
For $v \in N^+_i$ and for any $i$, since 
\[
\langle (a_{i1} - h_i, \dots, a_{id} - h_i, h_i), v  \rangle \ge w_i =  \langle (a_{i1} - h_i, \dots, a_{id} - h_i, h_i),  (1,\dots,1,q) \rangle,
\] 
we have 
\[
\langle (a_{i1} - h_i, \dots, a_{id} - h_i, h_i), v - (1,\dots,1,q) \rangle \ge 0.
\] 
Hence $v -  (1,\dots,1,q)$ is in $\mathbb{R}_{\ge 0} \phi({{C}})$ by \Cref{Halfspaces}. In other words, $\cap \; N^+_i = (1,\dots,1,q) + \mathbb{R}_{\ge 0} \phi({{C}})$. Since $(1,\dots,1,q) \in \relint ({{C}})$, we have $(\cap \; N^+_i) \cap \phi(C) \subseteq \relint ({{C}})$. \\

By \cite[Theorem 6.3.5.(b)]{BH}, $\overline{R[It, t^{-1}]}$ is Gorenstein if and only if the relative interior of $\phi(C)$ is principal. The above paragraph shows that $(\cap \; N^+_i) \cap \phi(C)$ is principally generated by $(1,\dots,1,q)$ and contained in the relative interior of $\phi(C)$, and by \Cref{PartCan}, $(1,\dots,1,q)$ is a part of minimal generating set for $\relint \phi(C)$. Hence the relative interior of $\phi(C)$ is principally generated by $(1,\dots, 1, q)$ if and only if $(\cap \; N^+_i) \cap \phi(C) = \relint \phi(C)$. This proves the statement.
\end{proof}
\end{thm}

The following example illustrates the above theorem when there is exactly one bounded half space among the half spaces defining the Newton polyhedron of the ideal $I$. This will help one to understand and prove \Cref{HKU5.4} which is Theorem 5.6 in \cite{HKU11}.

\begin{exam}\label{normal:exam}
Let $R = \mathbb{C}[x,y,z]$ and $I = (x^2, y^2, z^4)$. Then the integral closure of $I$ is determined by the half spaces $\{ v \in \mathbb{R}^3 \mid \langle (2,2,1), v \rangle \ge 4\}$ and $\{ v \in \mathbb{R}^3 \mid \langle \ee_i, v \rangle \ge 0 \}$ for $i = 1,2,3$ where  $\{ \ee_i \}$ denote the standard bases of $\mathbb{R}^3$. Let $\nu$ be the valuation corresponding to the bounded half space $\{ v \in \mathbb{R}^3 \mid \langle (2,2,1), v \rangle \ge 4\}$; then one has $\nu(x) = 2, \nu(y) = 2$, and $\nu(z) = 1$. Let $\overline{R[It,t^{-1}]}$ and ${{C}}$ be the corresponding affine semigroup. Then $\phi({{C}})$ is determined by the following half spaces represented as a matrix
\begin{equation*}
M = \left[
\begin{array}{ c c c c c }
-2 & -2 & -3 & 4 \\ \hdashline
1  & 0 & 0 & 0\\ \hdashline
0 & 1 & 0 & 0\\ \hdashline
0 & 0 & 1 & 0 
\end{array}
\right]
\end{equation*}
in the sense that an exponent vector $(z_1,\dots,z_4) \in \phi({{C}})$ if and only if all the entries of $M (z_1,\dots,z_4)^{tr} \ge 0$. Here $M \ge c$ where $c \in \mathbb{R}$ if all the entries of $M$ is greater than or equal to the number $c$. 
One can easily check that $(1,1,1,2) \in \phi({{C}})$, but $(1,1,1,1) \notin \phi({{C}})$. Hence $q = 2$. Furthermore, $
M \cdot (1,1,1,2)^{tr} = (1,1,1,1)^{tr}.$\\

For simplicity, we view elements of vector spaces as column vectors. If $w \in \relint \phi({{C}})$, then $M \cdot w \gneq 0$, i.e., $M \cdot w \ge 1$. 
Since $M \cdot w \ge 1$ and $M \cdot (1,1,1,2)^{tr}  = (1,1,1,1)^{tr}$, we have  $M \cdot w - M \cdot (1,1,1,2)^T \ge 0$, i.e., $M \cdot (w - (1,1,1,2)^{tr}) \ge 0$. Hence $w - (1,1,1,2)^{tr} \in \phi({{C}})$. This implies the exponent vector $(1,1,1,2)$ generates $\relint \phi({{C}})$. Therefore, $\overline{R[It,t^{-1}]}$ is Gorenstein by \Cref{normal:main}. Furthermore, by \cite[Corolllary 6.3.6]{BH} the ideal corresponding to the $\relint \phi({{C}})$ is the graded canonical ideal of $\overline{R[It,t^{-1}]}$, and it is generated by $xyz t$ which corresponds to $(1,1,1,2)$. 
This shows that $a = a(\overline{R[It,t^{-1}]}) = -1 = 2 - (1+1+1) = q - d$ where $d = \dim R = 3$. Since $\overline{R[It,t^{-1}]}$ is Cohen-Macaulay, $a = r - d + 1$ where $r = r(\mathcal{F})$ the reduction number of the filtration $\mathcal{F} = \{ \overline{I^i} \}_{i  \in \mathbb{Z}}$. 
Therefore, $r = -1 + 3 - 1 = 1$. 
\end{exam}

\begin{cor}
With the setting of \Cref{normal:main}, one has $a(\overline{R[It,t^{-1}]}) \ge q - d$ and $r(\mathcal{F}) \ge q - 1$. Furthermore, if $\overline{R[It,t^{-1}]}$ is Gorenstein, then $a(\overline{R[It,t^{-1}]}) = q - d$ and $r(\mathcal{F}) = q - 1$. 
\end{cor}

\begin{proof}
Notice that $a(\overline{R[It,t^{-1}]}) = r(\mathcal{F}) - d + 1$ since $\overline{R[It,t^{-1}]}$ is Cohen-Macaulay. The inequality $a(\overline{R[It,t^{-1}]}) \ge q - d$ shows the first part of the statement. The second part follows from the fact that $a(\overline{R[It,t^{-1}]}) = r - d  + 1$ if $\overline{R[It,t^{-1}]}$ is Gorenstein. 
\end{proof}

\begin{cor}[{\cite[Theorem 5.6]{HKU11}}]\label{HKU5.4}  
With the setting of \Cref{normal:main}, assume that $I = (x_1^{a_1}, \dots, x_d^{a_d})$. Let $L = \lcm (a_1,\dots, a_d)$. 
Write $L/ a_1 + \cdots + L / a_d = jL + p$, where $j \ge 0$ and $1 \le p \le L$. Then  $\overline{R[It, t^{-1}]}$ is Gorenstein if and only if $p = 1$. 
\begin{proof}
Let $\{ \ee_i \}$ be the standard basis of $\mathbb{R}^d$.
Let $H^+_i = \{ v \in \mathbb{R}^{d} \mid \langle  \ee_i, v \rangle \ge 0 \} $ for $i = 1,\dots,d$. 
Observe that the half spaces which determine the integral closure of $I$ are the bounded half space $\{ v \in \mathbb{Z}^d \mid \langle (L/a_1, \dots, L/a_d) , v \rangle \ge L \}$ and the $H^+_i$'s, and the bounded half space corresponds to the Rees valuation of $I$. 
Now, we proceed as in \Cref{normal:exam}. The affine semigroup $\phi({{C}})$ is determined by the following half spaces represented as a matrix
\setlength\dashlinedash{1pt}
\begin{equation*}
M = 
\left[\,
\begin{array}{ c c c c c }
L/a_1 - L & L/a_2 - L & \cdots & L/a_d - L & L \\ \hdashline
1  & 0 & \cdots & 0 & 0\\ \hdashline
0 & 1 & \cdots & 0 &0 \\ \hdashline
\vdots & \vdots & \vdots & \vdots & \vdots \\ \hdashline
0 & 0 &  0 & 1 & 0
\end{array}\,\right]
\end{equation*}
,i.e., $
\phi(\CC) = \{ (x_1,\dots, x_{d+1}) \in \ZZ_{\ge 0}^{d+1} \mid M \cdot (x_1,\dots, x_{d+1}) \ge 0 \}.$  
Therefore one has 
\begin{equation*}
\relint \phi(\CC) = \{ (x_1,\dots, x_{d+1}) \in \ZZ^{d+1}_{\ge 0} \mid M \cdot (x_1,\dots, x_{d+1}) \gneq 0 \}.
\end{equation*}
Let $q$ be as in \Cref{PartCan}. 
Then one has
$(\cap \; N_i^+) \cap \phi(C) = \{ (x_1,\dots, x_{d+1}) \in \ZZ^{d+1}_{\ge 0} \mid M \cdot (x_1,\dots, x_{d+1}) \ge M \cdot (1,\dots,1,q) \}.$
By \Cref{normal:main} $R[It,t^{-1}]$ is Gorenstein if and only if $\relint \phi(\CC) \subseteq \cap N_i^+$ that is  
\begin{align*}
 \{ (x_1,\dots, x_{d+1}) &\in \ZZ^{d+1}_{\ge 0 } \mid M (x_1,\dots, x_{d+1})^{tr} \gneq 0 \} \\
 &\subseteq 
 \{ (x_1,\dots, x_{d+1}) \in \ZZ^{d+1}_{\ge 0} \mid M  (x_1,\dots, x_{d+1})^{tr} \ge M  (1,\dots,1,q)^{tr} \}.
 \end{align*}

Observe that 
\begin{equation*}
M 
\begin{bmatrix}
x_1\\
x_2\\
\vdots  \\
x_{d+1}\\
\end{bmatrix} 
= 
\begin{bmatrix}
\frac L{a_1}x_1  + \dots + \frac L {a_d}x_d + L(x_{d+1} - (x_1 + \dots + x_d )) \\ 
x_1\\
\vdots \\
x_d
\end{bmatrix}
~\text{and}~
M 
\begin{bmatrix}
1\\
\vdots  \\
1\\
q\\
\end{bmatrix} 
= 
\begin{bmatrix}
\frac L{a_1}  + \dots + \frac L{a_d} + L(q - d) \\ 
1\\
\vdots \\
1
\end{bmatrix}.
\end{equation*}
Hence for any $(x_1, \dots, x_{d+1})$ in the set $\{ (x_1,\dots, x_{d+1}) \in \ZZ^{d+1}_{\ge 0 } \mid M (x_1,\dots, x_{d+1})^{tr} \gneq 0 \}$, we have that $x_i \ge 1$ for $1 \le i \le d$. Let $\rho := \min \{ \frac L{a_1}x_1  + \dots + \frac L {a_d}x_d + L(x_{d+1} - (x_1 + \dots + x_d ))   \mid \text{for integers}~x_1, \dots, x_d \ge 1~\text{and}~ x_{d+1} \ge 0 \}$ and write $\eta = \frac L{a_1}  + \dots + \frac L{a_d} + L(q - d)$. 
It suffices to show that $\eta \le \rho$ if and only if $p = 1$.  
Since $(1,\dots,1,q)$ is in $\relint (\phi(\CC))$, one has $\eta \gneq 0$, i.e., $\eta \ge 1$. Assume that $\rho = 1$. Then $ \relint \phi(\CC) \subseteq \cap N_i^+$ if and only if $\eta = 1$, and the last condition is equivalent to the condition of $p = 1$. 
Hence it suffices to show that $\rho = 1$. \\

First, we claim that $\gcd(L/a_1 - L , L/a_2 - L , \cdots , L/a_d - L , L) = 1$. Recall that $L = \lcm (a_1,\dots,a_d)$. Since $\delta := \gcd(L/a_1 - L , L/a_2 - L , \cdots , L/a_d - L , L) = \gcd (L/a_1,  \cdots , L/a_d, L) =  \gcd (L/a_1,  \cdots , L/a_d)$, we have $\delta | (L / a_i)$ for all $i$. This implies that $a_i | (L/ \delta)$  for all $i$ since $\delta$ divides $L$. 
Hence we see that  $L/\delta \ge \lcm (a_1,\dots,a_d) = L$ and this implies that $\delta = 1$. 
Next, we claim that there exists $(y_1,\dots,y_{d+1}) \in \mathbb{Z}^{d+1}_{\ge 0}$ where $y_i > 0$ for $i = 1, \dots, d$ such that $\langle (L/a_1 - L, L/a_2 - L, \cdots, L/a_d - L, L), (y_1, \dots, y_{d+1}) \rangle = 1$. 
Since $\gcd (L/a_1 - L , L/a_2 - L , \cdots , L/a_d - L , L) = 1$, there exist such $(y_1, \dots, y_{d+1})$ in $\ZZ^{d+1}$. 
We show that one can modify the $y_i$'s so that $y_i > 0$ for all $i = 1, \dots, d$. 
Suppose that $i \le d$ is the least index where $y_i \le 0$. Let $n$ be an integer such that $y_i + nL > 0$. By replacing $y_i$ by $y_i + nL$ and $y_{d+1}$ by $y_{d+1} - n(L/ a_i - L)$, we may assume that $y_i > 0$. 
This shows that the first row of $M \cdot (y_1, \dots, y_{d+1})^{tr}$ is $1$ and all the other rows are positive. This indeed shows that $\rho = 1$ and this completes the proof. 
\end{proof}
\end{cor}

\begin{rmk} 
One may ask if the numbers $w_i$ in \Cref{normal:main} are the minimum in $\{ \langle (a_1 - h_i, \dots, a_d - h_i, h_i), v \rangle \mid v \in \mathbb{Z}^{d+1}_{\ge 0} \}$ for each $i$. The following example shows that it can happen that $w_i$ are not the minimum for each halfspace, but it is the minimum in the intersection, i.e., it is the minimum in the relative interior. \\
Let $R = \mathbb{C}[x,y]$. Let $I = (x^{3},x y,y^{3})$. Let $T = \overline{R[It, t^{-1}]}$. Then $T$ is Gorenstein, and $\phi({{C}}_T)$  is determined by the half spaces
\[
\begin{bmatrix} 
-1 & -2 & 3  \\
-2  & -1 & 3 \\
1 & 0 &  0 \\ 
0 & 1 & 0 
\end{bmatrix}.
\]
One can easily see that $q = 2$. But $\langle (-1,-2,3),(3,1,2) \rangle = 1$ whereas $w_1 = \langle (-1,-2,3),(1,1,2) \rangle = 3$.\\ 
\end{rmk}

\section{The Cohen-Macaulayness of the associated graded ring}\label{sec:grDomain}
Let $(R,\m)$ be a regular local ring and $I$ an ideal. Let $J$ be a proper ideal. We define a function $\ord_{J}: R \rightarrow \mathbb{N}_{\ge 0} \cup \{ \infty \}$ as follows; $\ord_{J} (x) := \sup \{ i \mid x \in J^i \}$ for $x \in R$ and $\ord_J (I) := \inf \{ \ord_J(x) \mid x \in I \}$. By Krull's intersection theorem, this number is finite if $x \neq 0$. In general, we have $\ord_{\m}(xy) \ge \ord_{\m}(x)  + \ord_{\m}(y)$. Since $\gr_{\m} (R)$ is a domain, $\ord_{\m}$ is a valuation, i.e., $\ord_{\m}(xy) = \ord_{\m}(x)  + \ord_{\m}(y)$ for $x,y \in R$. 
For $x \in R$, let $x^* := x + \m^{\ord_{\m}(x)+1} \in \m^{\ord_{\m}(x)} / \m^{\ord_{\m}(x)+1} $ denote its image in $\gr_I(R)$. We call $x^*$ the \textit{leading form} of $x$ and $I^* := ( x^* \mid x \in I )$ the \textit{leading ideal} of $I$, respectively.

\begin{rmk}\label{Completetion}
Since $(R,\m)$ is a regular local ring, so is $(\widehat{R}, \widehat{\m})$ where $\widehat{\phantom{x}}$ denotes $\m$-adic completion. Because $\gr_{\m} (R) = \gr_{ \widehat{\m}} (\widehat{R})$, we have $\ord_{\m}(x) = \ord_{\widehat{\m}} (x)$ and $\ord_{\m}(I) = \ord_{\widehat{\m}}(I \widehat{R})$. In particular, we have $I^* = (\widehat{I})^*$ in $\gr_{\m} (R) = \gr_{ \widehat{\m}} (\widehat{R})$. Since $\widehat{R}$ is complete, we may write $x = \sum_{ i \ge 0 }^\infty [x]_i = \sum_{ i \ge \ord_{\m}(x) }^\infty [x]_i$ where $[x]_i  \in \widehat{\m}^i \setminus \widehat{\m}^{i+1} \cup \{ 0 \}$ and $x^* = ([x]_{\ord_{\m}(x)})^*$. 
\end{rmk}

\begin{lem}\label{S:ecom2} Let $(R,\m)$ be a regular local ring with maximal ideal $\m$. Let $I$ be an ideal which is minimally generated by the $2$ by $2$ minors of the $2$ by $3$ matrix $\displaystyle M = 
\begin{pmatrix} 
	a & b & c \\
	u & v & w
	\end{pmatrix}$ whose entries are in $\m$.
Write $G := \gr_{\m} (R)$. Let $L = ( b^* w^* - c^* v^*, -(a^* w^* - c^* u^* ),a^* v^* - b^* u^*) $ and let $C_\bullet$ be the complex 
\[
\begin{CD}
0 \rightarrow G^2 @>{\begin{bmatrix} a^* & u^* \\ b^* &   v^* \\ c^* & w^*   \end{bmatrix} }>> G^3 @>{ \small{ \begin{bmatrix} b^* w^* - c^* v^* & -(a^* w^* - c^* u^* ) & a^* v^* - b^* u^*   \end{bmatrix}  } } >> G.
\end{CD}
\]
If $L$ is of height $2$, then $C_\bullet$ is acyclic. Furthermore, if any two elements in the generating set of $L$ are homogeneous, then so is the third and $I^* = L = ((bw-cv)^*, (aw-cu)^*, (av-bu)^*)$. 
\begin{proof}
The acyclicity of $C_\bullet$ follows by the Buchbaum-Eisenbud acyclicity criterion \cite[Corollary 1]{BE73}. 
If any two elements of the generating set of the $L$ are homogeneous, then it is easy to see that the columns of the matrix representing the differential map $G^2 \to G^3$ have the same degree. Hence $L$ is homogeneous. 
Since $C_\bullet$ acyclic and all the entries of the maps are in the maximal homogenous ideal of $\gr_{\m} ( R)$, the ideal $(b^* w^* - c^* v^*, -(a^* w^* - c^* u^* ), a^* v^* - b^* u^*)$ is minimally generated by these three elements. In particular, these elements are not zero. This implies that $(bw-cv)^* = b^* w^* - c^* v^*, (aw-cu)^*  = a^* w^* - c^* u^*, (av-bu)^* = a^* v^* - b^* u^*$. \\

We are going to show that $I^* = L^*$.
Let $I = (f_1, f_2, f_3)$ where $f_1 = bw-cv, f_2 = -(aw-cu), f_3= av-bu$. By \Cref{Completetion} we may assume that $R$ is complete. By definition, $(f_1^*,f_2^*,f_3^*) \subseteq I^*$. Suppose $(f_1^*,f_2^*,f_3^*) \neq I^*$. Then there exists $x \in I$ such that $x^* \in I^* \setminus (f_1^*,f_2^*,f_3^*)$. Since $x \in I$, we can write $x = g_1f_1 + g_2f_3 +  g_3f_3$ for some $g_i \in R$. Since we are in a complete local ring, we can write $x = \sum_{i \ge 0}^\infty [x]_i$ as in \Cref{Completetion}. 
Observe that $\ord_{\m}(x) = \ord_{\m}( g_1f_1 + g_2f_2 +g_3 f_3)  \ge \min \{ \ord_{\m}(g_i f_i) \}_{i = 1,2,3}$. Since $x^* \not\in (f_1^*,f_2^*,f_3^*)$, the inequality is strict. We are going to show that this leads to a contradiction. \\

The set $\Gamma = \{ \min \{ \ord_{\m}(g_i f_i)_{i = 1,2,3} \} \mid \hbox{ for a triple } g_i \hbox{ such that } \, x = \sum g_i f_i \}$ is finite since for any $n \in \Gamma, n \le \ord_\m (x)$.  
We choose $g_i$'s such that the number $\min \{ \ord_{\m}(g_i f_i) \}_{i = 1,2,3}$ is the maximum in $\Gamma$.
We are going to construct $g'_i$'s such that $x =  g'_1f_1 + g'_2f_3 +  g'_3f_3$ and $\min \{ \ord_{\m}(g'_i f_i) \} \gneq \min \{ \ord_{\m}(g_i f_i) \}$. This will contradict the maximality. \\

Let $n := \min \{ \ord_{\m}(g_i f_i) \}$. If $\ord_{\m}(x) = n$, then $x^* \in (f_1^*,f_2^*,f_3^*)$. Therefore we may assume that $n < \ord_{\m}(x)$. 
Let $\delta_i := \ord_{\m}(f_i)$ for $i = 1,2,3$. We have 
$([g_1 f_1 + g_2 f_2 +  g_3 f_3]_n)^* = ([g_1]_{n - \delta_1})^* (f_1)^* + ([g_2]_{n - \delta_2})^* (f_2)^* + ([g_3]_{n - \delta_3})^* (f_3)^* = 0$. That is $([g_1]_{n - \delta_1})^* , ([g_2]_{n - \delta_2})^*, ([g_3]_{n - \delta_3})^*$ is a relation on $f_1^*, f_2^*, f_3^*$. Since the complex ${C}_\bullet$ is acyclic, there exists 
$s, t$ in $R$ such that 
\begin{equation*}
( ([g_1]_{n - \delta_1})^* , ([g_2]_{n - \delta_2})^*, ([g_3]_{n - \delta_3})^* )= s^* (a^*, b^*, c^*) + t^*( u^*, v^*, w^*) = (s^* a^*- t^* u^*, s^* b^*- t^* v^*,s^* c^*- t^* w^*).
\end{equation*}
Let $g'_1 = g_1 - (s a + t u), g'_2 = g_2 + (s b+t v)$, and $g'_3 = g_3 -  (s c + t w)$. 
Then $x = g_1 f_1 + g_2 f_2 + g_3 f_3 =  g'_1 f_1 + g'_2 f_2 + g'_3 f_3$ and $\min \{ \ord_{\m}(g_i f_i) \} \lneq \min \{ \ord_{\m}(g'_i f_i) \}$. This is a contradiction. 
\end{proof}
\end{lem}

We first state a couple of folklore remarks which will be useful for the proofs of \Cref{value1,oneforall}. 
\begin{rmk}\label{primeElement}
Let $(R,\m)$ be a Noetherian local ring. Let $x,y,u,v$ be elements in $\m$ such that $xy-uv \neq 0$. If $(xy-uv)^*$ is a prime element in $\gr_{\m}(R)$, then $\ord_{\m} (xy) = \ord_{\m} (uv)$. 
\begin{proof}
Assume to the contrary $\ord_{\m} (xy) \neq \ord_{\m} (uv)$. Without loss of generality, we may assume that $\ord_{\m} (xy) < \ord_{\m} (uv)$. Then $(xy-uv)^* = (xy)^* = x^* y^*$. Since $x^*, y^*$ are in the maximal homogeneous ideal of $\gr_{\m} (R)$, neither of them is a unit. Therefore the product $x^* y^*$ is not a prime element. This is a contradiction.
\end{proof}
\end{rmk}

\begin{rmk}\label{ordIne}
Let $(R,\m)$ be a Noetherian local ring. Let $x,y$ be elements in $\m$. If $\ord_{\m}(x) = \ord_{\m} (y)$ and $x^* + y^* \neq 0$, then $(x+y)^* = x^* + y^*$. In particular, $\ord_{\m}(x+y) = \ord_{\m} (x)$.
\begin{proof}
Let $G = \gr_{\m}(R)$. Recall that $\ord_{\m} (x+y) = \ord_{G_+} ((x+y)^*)$ 
where $G_+ = \oplus_{i > 0} G_i$. We have $\ord_{G_+} (x^* + y^*) \ge \min \{ \ord_{G_+} (x^*), \ord_{G_+} (y^*) \}$. If this is a strict inequality, then $(x+y)^* = x^*$ or $y^*$, and this is a contradiction. Also if $\ord_{G_+} (x^*) \neq \ord_{G_+} (y^*)$, then $(x+y)^* = x^*$ or $y^*$, and this is a contradiction.
\end{proof}

\end{rmk}

\begin{lem}\label{value1}
Let $(R,\m)$ be a regular local ring with maximal ideal $\m$. Let $I$ be an ideal of height $2$ which is minimally generated by the $2 \times 2$ minors of the $2 \times 3$ matrix $M$ with entries in $\m$. Assume that the leading forms of any two minors form part of a minimal generating set of $I^*$. If $I^*$ is a prime ideal and $I_1(M) \nsubseteq \m^2$, then we can find a matrix 
\[
\widetilde{M} = 	\left( 
	\begin{array}{ccc}
	a  & b  & c  \\
	u  & v  & w 
	\end{array} \right)
\]
such that $I = I_2(\widetilde{M})$ and  $I^* = ((b w -c v )^*, (a w -c u )^*, (a v -b u )^*)$ is perfect of height $2$. 
\begin{proof}
Write 
\[
M = 	\left( 
	\begin{array}{ccc}
	a & b & c \\
	u & v & w
	\end{array} \right).
\]
By switching rows and columns, which preserves the minors up to sign, we may assume that $a$ is in $\m \setminus \m^2$. Henceforth we will only use the assumption that $(av-bu)^*, (aw-cv)^*$ form part of a minimal generating set of $I^*$. 
We are going to modify $M$ by applying row and column operations to obtain $\widetilde{M}$ with the desired properties. 
Observe that adding a multiple of a row to the other does not change $2 \times 2$ minors, but adding a multiple of a column to another changes one minor, but does not change the other two in general. In both cases the ideal $I_2(-)$ does not change. In the proof the only type of column operations we perform is adding a multiple of the first column to the second or the third column. This does not change the minors $av-bu, aw-cu$, but the minor $bw-cv$ will be changed to $bw-cw + f(av-bu) + g( aw-cu)$ for some $f, g$ in $R$. \\

We claim that after performing row and column operations on $M$, we may assume that $(av-bu)^* = a^*v^* - b^* u^*$. Assume to the contrary that $(av-bu)^* \neq a^*v^* - b^* u^*$. Since $(av-bu)^*$ is part of a minimal generating set of a prime ideal $I^*$, $(av-bu)^*$ is a prime element. By \Cref{primeElement} we have $\ord_{\m} (av) = \ord_{\m} (bu)$.
Therefore, by \Cref{ordIne} one has $a^*v^* - b^* u^* = 0$ equivalently  $a^*v^* = b^* u^*$. Since $\gr_{\m} (R)$ is a UFD and $a^*$ is of degree $1$, it is a prime element. Hence we have $a^* | b^*$ or $a^* | u^*$. Suppose $a^* | b^*$.  Then there exists $\delta$ in $R$ such that $a^* \delta^* = b^*$, and this implies $\ord_{\m}(b) \lneq \ord_{\m}(b - \delta a)$.
We subtract the first column multiplied by $\delta$ from the second column to obtain
\[
M' = \left( 
	\begin{array}{ccc}
	a & b-\delta a & c \\
	u & v-\delta u & w
	\end{array} \right).
\]  
Since $\ord_{\m}(b) \lneq \ord_{\m}(b- \delta a)$, we have $ \ord_{\m}(ub) \lneq \ord_{\m}( u(b- \delta a))$. 
We replace $M$ by $M'$. The column operation changes the minors $av-bu,aw-cu,bw-cv$ to $av-bu,aw-cu,bw-cv+ \delta(aw-cu)$. 
If $a^* | u^*$, then we perform a row operation to obtain new $u, v, w$. We note that the row operation does not change the minors. 
We claim that this process terminates. Each time we replace $M$ by $M'$, either $\ord_{\m}(b)$ or $\ord_{\m}(u)$ strictly increases whereas $\ord_{\m}(av-bu)$ is fixed.
By \Cref{primeElement}, we have $\ord_{\m}(av) = \ord_{\m}(bu)$, 
and this implies $\ord_{\m}(av-bu) \ge  \min \{ \ord_{\m}(av), \ord_{\m}(bu) \} = \ord_{\m}(bu)$. 
The number $\ord_{\m}(av-bu)$ is fixed whereas $\ord_{\m}(bu)$ is strictly increasing after each process. Therefore this will terminate, and we obtain $(av-bu)^* = a^* v^* -b^* u^*$. \\ 

We are going to show that by subtracting a multiple of the first column from the third column, we can obtain a matrix $M$ such that $(av-bu)^* = a^* v^* -b^* u^*, (aw-cu)^* = a^* w^* - c^* u^*$. 
In particular, this will not change the entries $a,b,u,v$ of the matrix $M$; hence we preserve the property $(av-bu)^* = a^* v^* -b^* u^*$. 
Suppose $(aw-cu)^* \neq a^* w^* - c^* u^*$. Since $(aw-cu)^*$ is part of a minimal generating set of $I^*$, it is a prime element. By \Cref{primeElement} we have $\ord_{\m}(aw) = \ord_{\m} (cu)$  and $a^* w^* = c^* u^*$. If $a^* | u^*$, then the prime element $(av-bu)^* = a^* v^* -b^* u^*$ is divisible by $a^*$. This contradicts the fact that $a^*v^*-b^*u^*$ is a prime element. 
Therefore we have $a^* | c^*$.  
Then there exists $\delta$ in $R$ such that $a^* \delta^* = c^*$. We subtract the first column multiplied by $\delta$ from the third column. In particular, this does not change the entries $a,b,u,v$ of $M$. This process terminates, and we have $(av-bu)^* = a^* v^* -b^* u^*$ and $(aw-cu)^* = a^* w^* - c^* u^*$. \\

By \Cref{S:ecom2} it suffices to show that $\he (a^* w^* - c^*u^*, a^*v^* - b^*u^*) = 2$. The images $(av-bu)^*,(aw-cu)^*$ form a part of a minimal generating set of $I^*$. This implies that $((av-bu)^*)$ is a prime ideal in $\gr_{\m}(R)$ and $(aw-cu)^* \notin ((av-bu)^*)$. Hence height of the ideal $\he ((av-bu)^*,(aw-cu)^*)$ is $2$. Indeed our new $M$ is $\widetilde{M}$ in the statement. 
\end{proof}
\end{lem}

\begin{lem}\label{oneforall}Let $(R,\m)$ be a regular local ring with maximal ideal $\m$. Let $I$ be an ideal of height $2$ which is minimally generated by the $2$ by $2$ minors of the $2$ by $3$ matrix $M$ with entries in $\m$ where 
\[
M = 	\begin{pmatrix} 
	a & b & c \\
	u & v & w
	\end{pmatrix}.
\]
Assume that $(av-ub)^*,(aw-cu)^*$ is part of a minimal generating set of $I^*$. 
If $I^*$ is a prime ideal and $(av-bu)^* = a^*v^*-b^*u^*$, then we have $I^* = ((bw-cv + f(av-bu))^*, (aw-cu)^*, (av-bu)^*)$ for some $f$ in $R$, and this ideal is perfect of height $2$. 
\begin{proof}
Suppose $(aw-cu)^* \neq a^*w^*-c^*u^*$. 
Since $(av-bu)^*, (aw-cu)^*$ form part of a minimal generating set of a prime ideal, they are prime elements. By \Cref{primeElement} we have $\ord_{\m} (aw) = \ord_{\m} (cu)$. Therefore we have $a^*w^*-c^*u^* = 0$. Since $a^*v^*- b^*u^*$ is a prime element and $\gr_{\m}(R)$ is a UFD, $\gcd (a^*, u^*) \sim 1$. This implies that $a^* | c^*$ and $u^* | w^*$. Hence we may write $c^* = \delta^* a^*$ and $w^* = \delta^* u^*$ for some $\delta$ in $R$. Let $M'$ be a matrix modified by subtracting the first column of $M$ multiplied by $\delta$ from the third column. This column operation does not change the entries $a,b,u,v$ of the matrix $M$. Notice that $\ord_{\m}(c-\delta a) \gneq \ord_{\m}(c)$ and $\ord_{\m}(w - \delta u) \gneq \ord_{\m} (w)$. As in the proof of \Cref{value1}, this process will terminate. Hence $(aw-cu)^* = a^*w^*-c^*u^*$. Notice that the $2 \times 2$ minors of $M'$ are $av-bu, aw-cu, bw-cv + f(av-bu)$ for some $f$ in $R$.\\

 Since $(av-ub)^*,(aw-cu)^*$ form a part of a minimal generating set of $I^*$ and $(av-ub)^*$ is a prime element, we have $\he ((av-ub)^*, (aw-cu)^*) = 2$. The result follows by applying \Cref{S:ecom2} to the matrix $M'$.
\end{proof}
\end{lem}

\begin{thm}\label{thm:integralGr}
Let $S \cong R/I$, where $(R,\m)$ is a regular local ring and $I$ is a height $2$ perfect ideal. Assume that $\gr_{\n}( S)$ is an integral domain where $\n = \m /I$. If $\mu(I) \le 2$ or $\mu(I) = 3$ and $I \nsubseteq \m^5$, then $\gr_{\n} ( S)$ is Cohen-Macaulay.   
\begin{proof} 
Since $\gr_{\n}(S) \cong \gr_{\m} ( R) / I^*$ \cite[Exercise 5.3]{Ei95}, it suffices to show that $I^*$ is a Cohen-Macaulay ideal. There exists a minimal generating set of $I$ such that its leading forms are part of a minimal generating set of $I^*$. We fix a generating set of $I$ with this property. \\

Case 1: When $\mu(I) = 2$: Let $I = (f,g)$ for some $f$ and $g$ in $S$. We claim that $I^* = (f^*, g^*)$. Since $I^*$ is a prime ideal and $f^*, g^*$ form a part of minimal generating set of $I^*$, $f^*$ is a prime element. Since $f^*,g^*$ are part of a minimal generating set of $I^*$, $g^* \notin (f^*)$. Since $G/(f^*)$ is a domain, the image of $g^*$ in this ring is a non-zerodivisor. By \cite[Exercise 5.2]{Ei95}, the image of $I^*$ is generated by the image of $g^*$. Hence $I^* = (f^*, g^*)$. \\

Case 2: When $\mu(I) = 3$:  Since $I$ is a height 2 perfect ideal, by  the Hilbert-Burch theorem \cite{BH}, $I$ can be generated by the $2$ by $2$ minors of a $2$ by $3$ matrix $M$ where the minors are the chosen generators. Write
\[
M = \left( 
	\begin{array}{ccc}
	a & b & c \\
	u & v & w
	\end{array} \right).
\]

If $\ord_{\m}(I) \le 3$, $I_2(M) = I \nsubseteq \m^4$, hence $I_1(M) \nsubseteq \m^2$. Now, the result follows by applying \Cref{value1} to the matrix $M$.\\

Suppose $\ord_{\m}(I) = 4$. If there exists an entry of $M$ which has order $1$, then we may apply \Cref{value1}. We may assume that no entry of $M$ has order $1$. Since $\ord_{\m}(I) = 4$, at least one of the $2 \times 2$ minors of $M$ has order $4$. Without loss of generality, we may assume that $\ord_{\m}(av-bu) = 4$. By the assumption on orders, $\ord_{\m}(a),\ord_{\m}(b), \ord_{\m}(u)$, and $\ord_{\m}(v)$ are greater than or equal to $2$. 
Since $4 = \ord_{\m}(av-bu) \ge \min \{ \ord_{\m}(av), \ord_{\m}(bu) \} \ge 4$, we have $\ord_{\m}(a) = \ord_{\m}(b) = \ord_{\m}(u) = \ord_{\m}(v) = 2$. This implies $(av-bu)^* =a^* v^* - b^* u^*$. Then we are done by \Cref{oneforall}.
\end{proof}
\end{thm}

\begin{rmk}
One can not relax the condition of $\gr_I(R)$ a domain. Let $R = \mathbb{C}[[ a,b,c,d,e ]]$ and $I$ be the $2$ by $2$ minors of the matrix $M$,
\[
\begin{bmatrix}
a^2+c^3 & 0 & ad+c^3 \\
ab+c^3 & ae + a^3 & 0
\end{bmatrix}.
\]
Then $I^*$ is not prime, and $G/I^*$ is not Cohen-Macaulay. 
\end{rmk}

One can see from \Cref{value1,oneforall} that once we can find a minor which commutes with taking $^*$, then we can find a matrix $M$ where the images of the minors generate the leading ideal. The following theorem analyzes the case when none of the minors commute with taking $^*$. 

\begin{thm}
Let $(R,\m)$ be a regular local ring with maximal ideal $\m$. Let $I$ be an ideal which is minimally generated by the $2$ by $2$ minors of a $2$ by $3$ matrix $M$ with entries in $\m$ where 
\[
M = 	\begin{pmatrix} 
	a & b & c \\
	u & v & w
	\end{pmatrix}.
\]
Let 
\[
M^* = \begin{pmatrix} 
	a^* & b^* & c^* \\
	u^* & v^* & w^*
	\end{pmatrix}.
\]
Suppose that $(av-bu)^*, (aw-cu)^*$ form part of a minimal generating set of $I^*$ and $I^*$ is a prime ideal. If $(av-bu)^* \neq a^*v^* - b^*u^*$ and $(aw-cu)^* \neq a^* w^* - c^* u^*$, then either $\he I_1(M^*) \le 2$ or one of the rows of $M^*$ divides the other. 
\begin{proof}
Since the images of the minors form part of a minimal generating set of an prime ideal $I^*$, we have $a^*v^* = b^*u^*, a^* w^* =  c^* u^*$. 
Let $p^* = \gcd (a^*,u^*)$ where $p$ in $R$. Write $a^* = p^* (a')^*, u^* = p^* (u')^*$ where $a', u'$ in $R$. 
Then $(a')^* v^* = b^*(u')^*$ where $\gcd ((a')^*, (u')^*) \sim 1$. 
Therefore $(a')^* | b^*$ and $(u')^* | v^*$. Let $q$ in $R$ such that $(a')^* q^* = b^*$ and $(u')^* q^* = v^*$. From $a^* w^* =  c^* u^*$, we have $(a')^* w^* =  c^* (u')^*$. Therefore there exists $r$ in $R$ such that $(a')^* r^* = c^*$ and $(u')^* r^* = w^*$. Now, we have
\[
M^* = \begin{pmatrix} 
	(a')^* p^* & (a')^* q^* & (a')^* r^* \\
	(u')^* p^* & (u')^* q^* & (u')^* r^*
	\end{pmatrix}.
\]
Therefore $I_1(M^*) \subseteq ((a')^*,(u')^*)$, and this implies $\he I_1(M^*) \le 2$ if $((a')^*,(u')^*)$ is not a unit ideal. If $((a')^*,(u')^*)$ is a unit ideal, then either $(a')^*$ or $(u')^*$ is a unit. Without loss of generality, assume that $(a')^*$ is a unit. Then one can easily see that indeed the first row divides the second row. 
\end{proof}
\end{thm}

\section{Serre's conditions}\label{sec:serre}
In this section we show that when a ring $R$ is local, equidimensional, and universally catenary, if $\gr_I(R)$ satisfies Serre's condition $(S_i)$ (or $(R_i)$), then $R[It,t^{-1}]$ satisfies Serre's condition $(S_i)$ (or $(R_i)$), and $R$ satisfies Serre's condition $(S_i)$. \\

When $R$ is a $\mathbb{Z}$-graded ring and $\p$ is a prime ideal, let $\p^*$ denote  the ideal generated by all homogeneous elements in $\p$. This is a homogeneous prime ideal which has height exactly one less than that of $\p$ if $\p$ is not a homogenous ideal. 
Recall that a Noetherian ring $R$ satisfies \textit{Serre's condition $(S_i)$} if for every prime ideal $\p$ of $R$, $\depth R_\p \ge \min \{ i, \dim R_\p \}$.
A Noetherian ring $R$ satisfies \textit{Serre's condition $(R_i)$} if for every prime ideal $\p$ of $R$ with $\dim R_\p \le i$, the ring $R_\p$ is regular.

\begin{lem}[{\cite[Theorem 1.5.9 and Exercise 2.1.27, 2.2.24]{BH}}]\label{homogLocalization}
Let $R$ be a Noetherian $\mathbb{Z}$-graded ring.
\begin{enumerate}[$(a)$]
\item For $\p \in \Spec (R)$ the localization $R_\p$ is regular $($Cohen-Macaulay$)$ if and only if $R_{\p^*}$ is. 
\item Let $\p \in \Spec (R)$. If $\p$ is not homogeneous, then $\depth R_\p = \depth R_{\p^*} + 1$.  
\end{enumerate}
\end{lem}

\begin{thm}\label{thm:Serre} Let $(R, \m)$ be a local equidimentional universally catenary ring. Let $I \subseteq \m$ be an $R$-ideal. Consider the following conditions:
\begin{enumerate}[$(a)$]
\item The ring $\gr_I(R)$ satisfies Serre's condition $(S_i)$ $($or $R_i)$.
\item The ring $R[It,t^{-1}]$ satisfies Serre's condition $(S_i)$ $($or $R_i)$.
\item The ring $R$ satisfies Serre's condition $(S_i)$.
\end{enumerate}
We have $(a) \Longrightarrow (b) \Longrightarrow (c)$.
\begin{proof}
\noindent (a) $\Longrightarrow$ (b): Let $\pi: R[It,t^{-1}] \longrightarrow \gr_I ( R)$ be the natural surjective ring homomorphism. By \Cref{homogLocalization}(a)(b) it suffices to show Serre's condition $(S_i)$ (or $(R_i)$) for homogenous prime ideals. Let $P \subseteq R[It,t^{-1}]$ a homogeneous prime ideal. Since $R$ is universally catenary and equidimensional, so is $R[It,t^{-1}]$. This implies that $\he (P + (t^{-1}) ) \le \he P + 1$. We can choose a minimal prime $Q$ of $P  + (t^{-1})$ of height $\he (P + (t^{-1}) )$.  
We first show the statement for Serre's condition $(S_i)$. 
Since $\gr_I(R)$ satisfies Serre's condition $(S_i)$, $\depth \gr_I(R)_{\pi(Q)} \ge \min \{i, \dim \gr_I(R)_{\pi(Q)} \}$. Since $\depth \gr_I(R)_{\pi(Q)} = \depth R[It, t^{-1}]_Q - 1$ and $\dim \gr_I(R)_{\pi(Q)} = \dim R[It,t^{-1}]_Q \allowbreak - 1$, we have 
\begin{equation}\label{eq:dimDepth}
\depth R[It, t^{-1}]_Q \ge \min \{i + 1, \dim R[It,t^{-1}]_Q \}.
\end{equation}
If $P = Q$, i.e., $t^{-1} \in P$, then we are done. Suppose $P \subsetneq Q$. We need to show that $\depth R[It, t^{-1}]_P \ge \min \{i , \dim R[It,t^{-1}]_P \}$. Since $\dim R[It,t^{-1}]_Q = \dim R[It,t^{-1}]_P - 1$, by \Cref{eq:dimDepth} it suffices to show that $\depth R[It,t^{-1}]_P \ge \depth R[It,t^{-1}]_Q + 1$. 
This follows immediately once we have shown that 
\[
\Ext^j_{R[It, t^{-1}]_Q}( (R[It,t^{-1}]/P)_Q, R[It, t^{-1}]_Q) = 0 \hbox{ for } j < \depth R[It,t^{-1}]_Q - 1.
\] 
Since $\dim (R[It,t^{-1}]/P)_Q = 1$, this follows from \cite[(15.E) Lemma 2]{Mat}.\\ 

Now, suppose that $\gr_I(R)$ satisfies Serre's condition $(R_i)$. 
When $\he P \le i$, since $\he \pi(Q) \le \he P \le i$, $\gr_I(R)_{\pi(Q)}$ is regular. Since $R[It,t^{-1}]/ (t^{-1}) \cong \gr_I (R)$ and $t^{-1}$ is a regular element,  $R[It,t^{-1}]_Q$ is regular. Since $P \subseteq Q$, $R[It,t^{-1}]_P$ is regular. \\ 
\noindent (b) $\Longrightarrow$ (c): Let $\p$ be a prime ideal of $R$ of height $c$.  Recall that $ R[It,t^{-1}] / (t^{-1}-1) \cong R$. Let $P$ be the pre-image of $\p$ in $R[It,t^{-1}]$. Since $R[It,t^{-1}]$ is equidimensional and universally catenary, $\he P = c + 1$. Since $P$ contains $t^{-1} -1$, it is a non homogeneous prime ideal of $R[It,t^{-1}]$. Therefore, $P^*$ is a homogeneous prime ideal of height $c$.
Since $R[It,t^{-1}]$ satisfies Serre's condition $(S_i)$, we have $\depth R[It,t^{-1}]_{P^*} \ge \min \{ i , \dim R[It,t^{-1}]_{P^*} \}$. Recall that $t^{-1} - 1$ is a regular element in $R[It,t^{-1}]$ and $R \cong R[It,t^{-1}]/ (t^{-1} - 1)$. We have  $\depth R[It,t^{-1}]_{P^*} = \depth R[It,t^{-1}]_{P} - 1 = \depth R_\p + 1  - 1 = \depth R_\p$ and $\dim R[It,t^{-1}]_{P^*} = \dim R[It,t^{-1}]_{P} - 1 = \dim R_\p + 1  - 1 = \dim R_\p$ where the first equality follows from \Cref{homogLocalization}(b). Therefore, $\depth R_\p \ge \min \{ i , \dim R_\p \}$. 
\end{proof}
\end{thm}

\end{document}